\documentclass[a4paper,12pt]{amsart}
\usepackage{amsmath,amsthm,amssymb,latexsym,epic}
\usepackage{graphicx,enumerate}

\newtheorem{theorem}{Theorem}
\newtheorem{lemma}[theorem]{Lemma}
\newtheorem{remark}[theorem]{Remark}
\newtheorem{corollary}[theorem]{Corollary}
\newtheorem{proposition}[theorem]{Proposition}
\newtheorem{example}[theorem]{Example}

\usepackage[all]{xy}

\newcommand{\tto}{\twoheadrightarrow}

\begin{document}
\title[Koszul duality for stratified algebras]{Koszul 
duality for stratified algebras II. Standardly stratified algebras}
\author{Volodymyr Mazorchuk}
\date{}

\maketitle

\begin{abstract}
We give a complete picture of the interaction between the Koszul 
and Ringel dualities for graded standardly stratified algebras (in the 
sense of Cline, Parshall and Scott)  admitting linear  tilting 
(co)resolutions of standard and proper costandard modules. We
single out a certain class of graded standardly stratified algebras,
imposing the condition that standard filtrations of projective
modules are finite, and  develop a tilting theory for such algebras.
Under the assumption on existence of linear tilting  (co)resolutions
we show that algebras from this class are Koszul, that both the Ringel 
and Koszul duals belong to the same class, and that these two dualities 
on this class commute.
\end{abstract}

\section{Introduction}\label{s1}

In the theory of quasi-hereditary algebras there are two classical
dualities: the Ringel duality, associated with the characteristic 
tilting module (see \cite{Ri}), and the Koszul duality, associated with
the category of linear complexes of projective modules (see
\cite{CPS2,ADL2,MO}). In \cite{MO,Ma5} it is shown that a certain class 
of Koszul quasi-hereditary algebras is stable with respect to taking 
both the Koszul and Ringel duals and that on this class of algebras
the Koszul and Ringel dualities commute. 

The approach of \cite{Ma5} is ultimately based on the possibility
to realize the derived category of our algebra as the homotopy category 
of complexes of tilting modules. This also suggested that the arguments of 
\cite{Ma5} should work in a much more general setup, whenewer
an appropriate stratification of the algebra and a sensible 
tilting theory with respect to this stratification exist. The aim of 
the present paper is to define a setup for the study of Koszulity
for stratified algebras and to extend to this setup the main result 
of \cite{Ma5}. We note that Koszul standardly stratified algebras, 
which are not quasi-hereditary, appear naturally in  
\cite{ADL3,Fr3,KKM}. 

The most general setup for stratified algebras seems to be the 
notion of standardly stratified algebras as introduced by
Cline, Parshall and Scott in \cite{CPS}. The main problem which one faces,
trying to generalize \cite{Ma5} to such stratified algebras, is that
standardly stratified algebras have infinite global dimension in general. 
In particular, this means that the Koszul dual
of such an algebra (in the case when the original algebra is Koszul) is
always infinite dimensional. Therefore any reasonable extension of
\cite{Ma5} to stratified algebras must deal with infinite dimensional
stratified algebras, for which many of the classical results are
not proved and lots of classical techniques are not developed.

In the present paper we study the class of positively graded 
standardly stratified algebras with finite dimensional 
homogeneous components satisfying the additional assumption that 
all projective modules have finite standard filtrations. For 
such algebras we develop an analogue of the classical tilting
theory and Ringel duality. This follows closely the classcal 
theory, however, at some places one has to be careful as we work
with infinite dimensional algebras, so some extension spaces
might be infinite dimensional. We use the grading to split these 
infinite dimensional spaces into an (infinite) direct sum of 
finite dimensional ones. We also give some examples which justify
our choice of algebras and show that outside the class we define
the classical approach to tilting theory fails. The Ringel duality
functor turns out to be an antiequivalence between three 
different kinds of derived categories.
 
Using the standard grading of a characteristic tilting module, 
we restrict our attention to those standardly stratified algebras,
for which all tilting coresolutions of standard modules and all
tilting resolutions of proper costandard modules are linear. 
For an algebra $A$ let $R(A)$ and $E(A)$ denote the Ringel and Koszul
duals of $A$, respectively. Generalizing the arguments of \cite{Ma5} 
we prove the following (see Section~\ref{s2} for the definitions):

\begin{theorem}\label{thm1}
Let $A$ be a positively graded standardly stratified algebra
with finite dimensional homogeneous components. Assume that 
\begin{enumerate}[(a)]
\item\label{thm1-c1} Every indecomposable projective $A$-module
has a finite standard filtration.
\item\label{thm1-c2} Every standard $A$-module
has a linear tilting coresolution.
\item\label{thm1-c3} Every costandard $A$-module
has a linear tilting resolution.
\end{enumerate}
Then the following holds:
\begin{enumerate}[(i)]
\item\label{thm1-1} The algebra $A$ is Koszul.
\item\label{thm1-2} The algebras $A$, $R(A)$, $E(A)$, 
$E(R(A))$ and $R(E(A))$ have properties \eqref{thm1-c1},
\eqref{thm1-c2} and \eqref{thm1-c3}.
\item\label{thm1-3} Every simple $A$-module is represented 
(in the derived category) by a linear complex of tilting modules.
\item\label{thm1-4} $R(E(A))\cong E(R(A))$ as graded 
standardly stratified  algebras. 
\end{enumerate}
\end{theorem}

Theorem~\ref{thm1} extends and generalizes results from 
\cite{ADL2,ADL3,MO,Ma5}.

The paper is organized as follows: in Section~\ref{s2} we collected
all necessary definitions and preliminaries. In Sections~\ref{s3}
and \ref{s31} we develop the tilting theory for graded standardly 
stratified algebras. This theory is used in Section~\ref{s4} to 
prove Theorem~\ref{thm1}. We complete the paper with several examples 
in Section~\ref{s5}.
\vspace{0.5cm}

\noindent
{\bf Acknowledgments.} The research was partially supported by the
Swedish Research Council. A part of the results in the paper were
obtained during the visit of the author to Department of Algebra and
Number Theory, E{\"o}tv{\"o}s University, Budapest in September 2008.
The hospitality of E{\"o}tv{\"o}s University is gratefully acknowledged.
The author also thanks Istv{\'a}n {\'A}goston and Erzs{\'e}bet Luk{\'a}cs
for their hospitality and many stimulating discussions. I am grateful
to the referee for the very careful reading of the paper, for pointing 
out several inaccuracies and a gap in the original version and for 
numerous suggestions which led to improvements in the original version
of the paper.

\section{Graded standardly stratified algebras}\label{s2}

By $\mathbb{N}$ we denote the set of all positive integers. 
By a grading we always  mean a {\em $\mathbb{Z}$-grading} and
by a module we always mean a {\em graded left} module. 

Let $\Bbbk$ be an algebraically closed field and 
$A=\bigoplus_{i\geq 0}A_i$ be a graded 
$\Bbbk$-algebra. We assume that $A$ is {\em locally finite},
that is $\dim_{\Bbbk}A_i<\infty$. Set 
$\mathrm{r}(A):=\bigoplus_{i>0}A_i$. We further assume that
$A_0\cong \bigoplus_{\lambda\in\Lambda}\Bbbk e_{\lambda}$ for some 
set $\{e_{\lambda}:\lambda\in\Lambda\}$ of pairwise orthogonal 
nonzero idempotents in $A_0$, where $\Lambda$ is a nonempty finite set
(using the classical Morita theory one extends all our results
to the case when $A_0$ is a semi-simple algebra). 
Under these assumptions the algebra $A$ is positively graded
in the sense of \cite{MOS}. In what follows we call $A$
{\em positively graded} if it satisfies all assumptions of this
paragraph. A typical example of a positively graded algebra is 
$\Bbbk[x]$, where $1$ has degree zero and $x$ has degree one.

Let $A\text{-}\mathrm{gmod}$ denote the category of all locally finite dimensional 
graded $A$-modules. Morphisms in this category are homogeneous morphism
of degree zero between graded $A$-modules. Consider the full subcategories
$A^{\uparrow}\text{-}\mathrm{gmod}$ and $A^{\downarrow}\text{-}\mathrm{gmod}$ 
of $A\text{-}\mathrm{gmod}$, which consist of all graded modules 
$M=\bigoplus_{i\in \mathbb{Z}}M_i$ for which there exists
$n\in\mathbb{Z}$ such that $M_i=0$ for all $i>n$ or $i<n$, respectively.
All these categories are abelian,  the category 
$A^{\downarrow}\text{-}\mathrm{gmod}$ has enough projectives and the
category $A^{\uparrow}\text{-}\mathrm{gmod}$ has enough injectives.
For $M\in A^{\downarrow}\text{-}\mathrm{gmod}$ we set
\begin{displaymath}
\mathfrak{b}(M)=
\begin{cases}
+\infty, & M=0; \\
\min_{n\in\mathbb{Z}}\{M_n\neq 0\}, & M\neq 0.
\end{cases} 
\end{displaymath}

For $i\in\mathbb{Z}$ we denote by $\langle i\rangle$ the autoequivalence 
of $A\text{-}\mathrm{gmod}$, which shifts the grading as follows: 
$(M\langle i\rangle)_j=M_{i+j}$, where $j\in \mathbb{Z}$. This autoequivalence
preserves both $A^{\uparrow}\text{-}\mathrm{gmod}$ and
$A^{\downarrow}\text{-}\mathrm{gmod}$. Denote by $\circledast$ the usual 
{\em graded duality} on $A\text{-}\mathrm{gmod}$
(it swaps $A^{\uparrow}\text{-}\mathrm{gmod}$ and
$A^{\downarrow}\text{-}\mathrm{gmod}$). We adopt the  notation 
$\mathrm{hom}_A$ and $\mathrm{ext}^i_A$ to denote homomorphisms 
and extensions in  $A\text{-}\mathrm{gmod}$. Unless stated otherwise, all
morphisms are considered in the category $A\text{-}\mathrm{gmod}$.

For $\lambda\in \Lambda$  we consider the graded indecomposable 
projective module $P(\lambda)=Ae_{\lambda}$, its graded simple 
quotient $L(\lambda)=P(\lambda)/\mathrm{r}(A)P(\lambda)$ and the 
graded indecomposable injective envelop $I(\lambda)$ of $L(\lambda)$.
Note that we always have the following:
$P(\lambda)\in A^{\downarrow}\text{-}\mathrm{gmod}$,
$I(\lambda)\in A^{\uparrow}\text{-}\mathrm{gmod}$
and $L(\lambda)\in A^{\downarrow}\text{-}\mathrm{gmod}
\cap A^{\uparrow}\text{-}\mathrm{gmod}$.

Let $\preceq$ be a partial preorder on $\Lambda$. For $\lambda,\mu\in\Lambda$
we write $\lambda\prec\mu$ provided that $\lambda\preceq\mu$ and
$\mu\not\preceq\lambda$. We also write $\lambda\sim\mu$ provided that 
$\lambda\preceq\mu$ and $\mu\preceq\lambda$. Then $\sim$ is an equivalence
relation. Let $\overline{\Lambda}$ denote the set of equivalence classes 
of $\sim$. Then the preorder $\preceq$ induces a partial order on
$\overline{\Lambda}$, which we will denote by the same symbol, abusing 
notation. For $\lambda\in\Lambda$ we denote by $\overline{\lambda}$ the
equivalence class from $\overline{\Lambda}$, containing $\lambda$.
We also denote by $\preceq^{\mathrm{op}}$ the partial preorder on 
$\Lambda$, opposite to $\preceq$.

For $\lambda\in\Lambda$ we define the {\em standard module}
$\Delta(\lambda)$ as the quotient of $P(\lambda)$ modulo the submodule,
generated by the images of all possible  morphisms 
$P(\mu)\langle i\rangle\to P(\lambda)$, where $\lambda\prec\mu$
and $i\in\mathbb{Z}$. We also define the 
{\em proper standard module} $\overline{\Delta}(\lambda)$ as the 
quotient of $P(\lambda)$ modulo the submodule, generated by the images 
of all possible morphisms  $P(\mu)\langle i\rangle\to P(\lambda)$, 
where $\lambda\preceq\mu$ and $i\in\mathbb{Z}$ satisfies $i<0$. By 
definition, the modules $\Delta(\lambda)$ and $\overline{\Delta}(\lambda)$
belong to $A^{\downarrow}\text{-}\mathrm{gmod}$. Dually we define
the {\em costandard module} $\nabla(\lambda)$ and the 
{\em proper costandard module} $\overline{\nabla}(\lambda)$
(which always belong to $A^{\uparrow}\text{-}\mathrm{gmod}$).

The algebra $A$ will be called {\em standardly stratified} 
(with respect to the preorder $\preceq$ on $\Lambda$) provided
that for every $\lambda\in\Lambda$ the kernel $K(\lambda)$ of the 
canonical  projection $P(\lambda)\tto \Delta(\lambda)$ has a {\em finite}
filtration, whose subquotients are isomorphic (up to shift) to
standard modules. This is a natural generalization of the original
definition from \cite{CPS} to our setup. For example, the algebra
$A$ is always standardly stratified (with projective standard modules)
in the case $|\Lambda|=1$ and, more generally, in the case
when the relation $\preceq$ is the full relation. 

\section{Tilting theory for graded standardly stratified algebras}\label{s3}

Tilting theory for (finite dimensional) quasi-hereditary algebras was 
developed in \cite{Ri}. It was extended in \cite{AHLU} to 
(finite dimensional) strongly standardly stratified algebras and in 
\cite{Fr} to all (finite dimensional) standardly stratified algebras. 
For infinite dimensional algebras some versions of tilting theory
appear in \cite{CT,DM,MT}. This section is a further generalization
of all these results, especially of those from \cite{Fr}, to the case
infinite dimensional positively graded algebras. In this section $A$ is a 
positively graded standardly stratified algebra.

Let $\mathcal{C}(\Delta)$ denote the full subcategory of
the category $A^{\downarrow}\text{-}\mathrm{gmod}$, which consists
of all modules $M$ admitting a (possibly infinite) filtration 
\begin{equation}\label{eqfiltration}
M=M^{(0)}\supseteq M^{(1)}\supseteq M^{(2)}\supseteq\dots,
\end{equation}
such that for every $i=0,1,\dots$ the subquotient 
$M^{(i)}/M^{(i+1)}$ is isomorphic (up to shift) to some standard module
and $\displaystyle\lim_{i\rightarrow+\infty}\mathfrak{b}(M^{(i)})=+\infty$.
Note that for $M\in A^{\downarrow}\text{-}\mathrm{gmod}$ with such a
filtration we 
automatically get $\displaystyle \bigcap_{i\geq 0}M^{(i)}=0$.
Denote by $\mathcal{F}^{\downarrow}(\Delta)$ the full subcategory of
$A^{\downarrow}\text{-}\mathrm{gmod}$, which consists
of all modules $M$ admitting a finite filtration with subquotients
from $\mathcal{C}(\Delta)$. The category $\mathcal{F}^{\downarrow}(\Delta)$
is obviously closed with respect to finite extensions.
Similarly we define $\mathcal{F}^{\downarrow}(\overline{\nabla})$.
Let $\mathcal{F}^{b}(\Delta)$ and $\mathcal{F}^{b}(\overline{\nabla})$
be the corresponding full subcategories of modules with finite 
filtrations of the form \eqref{eqfiltration}. We start with the 
following result, which  generalizes the corresponding results 
from \cite{AB,AR,Ri,Fr}.

\begin{theorem}\label{thm2}
Let $A$ be a positively graded standardly stratified algebra.
\begin{enumerate}[(i)]
\item \label{thm2-1} 
We have
{\small
\begin{displaymath}
\begin{array}{rcl} 
\mathcal{F}^{\downarrow}(\Delta)&=&
\{M\in A^{\downarrow}\text{-}\mathrm{gmod}\,:\,
\mathrm{ext}^i_A(M,\overline{\nabla}(\lambda)\langle j\rangle)=0,
\forall j\in\mathbb{Z},i>0,\lambda\in\Lambda\}\\
&=&
\{M\in A^{\downarrow}\text{-}\mathrm{gmod}\,:\,
\mathrm{ext}^1_A(M,\overline{\nabla}(\lambda)\langle j\rangle)=0,
\forall j\in\mathbb{Z},\lambda\in\Lambda\}.
\end{array}
\end{displaymath}
}
\item \label{thm2-2}
We have
{\small
\begin{displaymath}
\begin{array}{rcl} 
\mathcal{F}^{\downarrow}(\overline{\nabla})&=&
\{M\in A^{\downarrow}\text{-}\mathrm{gmod}\,:\,
\mathrm{ext}^i_A(\Delta(\lambda)\langle j\rangle,M)=0,
\forall j\in\mathbb{Z},i>0,\lambda\in\Lambda\}\\
&=&
\{M\in A^{\downarrow}\text{-}\mathrm{gmod}\,:\,
\mathrm{ext}^1_A(\Delta(\lambda)\langle j\rangle,M)=0,
\forall j\in\mathbb{Z},\lambda\in\Lambda\}.
\end{array}
\end{displaymath}
}
\end{enumerate}
\end{theorem}

To prove Theorem~\ref{thm2} we will need several auxiliary lemmata.
We will often use the usual induction for stratified algebras.
To define this let $\lambda\in \Lambda$ be maximal with respect to 
$\preceq$. Set
$e_{\overline{\lambda}}=\sum_{\mu\in\overline{\lambda}}e_{\mu}$,
$I_{\overline{\lambda}}=Ae_{\overline{\lambda}}A$ 
and $B_{\overline{\lambda}}=A/I_{\overline{\lambda}}$. The algebra
$B_{\overline{\lambda}}$ inherits from $A$ a positive grading and 
hence is a positively graded locally finite algebra.
Further, just like in the case of usual stratified algebras, the 
algebra $B_{\overline{\lambda}}$ is stratified with respect to the
restriction of the preorder $\preceq$ to 
$\Lambda\setminus\{\overline{\lambda}\}$.
Any module $M$ over $B_{\overline{\lambda}}$ can be considered as an
$A$-module in the usual way. Set 
$P(\overline{\lambda})=\bigoplus_{\mu\in \overline{\lambda}}P(\mu)$.

\begin{lemma}\label{lem3}
For all $M,N\in B_{\overline{\lambda}}^{\downarrow}\text{-}\mathrm{gmod}$
and all $i\geq 0$ we have
\begin{displaymath}
\mathrm{ext}^i_{B_{\overline{\lambda}}}(M,N)=
\mathrm{ext}^i_{A}(M,N).
\end{displaymath}
\end{lemma}

\begin{proof}
Let $\mathcal{P}^{\bullet}$ denote the minimal projective resolution of 
$M$ in $A^{\downarrow}\text{-}\mathrm{gmod}$. As 
$M\in B_{\overline{\lambda}}^{\downarrow}\text{-}\mathrm{gmod}$,
there exists $k\in\mathbb{Z}$ such that $M_j=0$ for all $j<k$.
Since $A$ is positively graded, we get $\mathcal{P}^{i}_j=0$
for all $j<k$ and all $i$. 

Consider the projective module 
$P=\bigoplus_{j\leq -k}P(\overline{\lambda})\langle j\rangle$. 
As $A$ is standardly stratified, for every $i$ the sum $T^i$
of images of all homomorphisms from $P$ to
$\mathcal{P}^{i}$ has the form $\bigoplus_{j\leq -k} P_{i,j}$, 
where $P_{i,j}\in\mathrm{add}P(\overline{\lambda})\langle j\rangle$. 

The differential of $\mathcal{P}^{\bullet}$ obviously maps 
$T^i$ to $T^{i-1}$, which means that the sum of all $T^i$ is a subcomplex
of $\mathcal{P}^{\bullet}$, call it $\mathcal{T}^{\bullet}$. 
Since $M\in B_{\overline{\lambda}}^{\downarrow}\text{-}\mathrm{gmod}$,
the quotient $\overline{\mathcal{P}}^{\bullet}$ of $\mathcal{P}^{\bullet}$ 
modulo $\mathcal{T}^{\bullet}$ gives a minimal projective resolution of
$M$ over $B_{\overline{\lambda}}$. 

Since $N\in B_{\overline{\lambda}}^{\downarrow}\text{-}\mathrm{gmod}$,
any homomorphism from $\mathcal{P}^{i}$ to $N$ annihilates
$\mathcal{T}^{i}$ and hence factors through 
$\overline{\mathcal{P}}^{i}$. The claim of the lemma follows.
\end{proof}

\begin{lemma}\label{lem4}
For all $\mu\in\Lambda$ we have  $\overline{\nabla}(\mu)\in 
A^{\downarrow}\text{-}\mathrm{gmod}$, in particular, 
$\overline{\nabla}(\mu)$ is finite dimensional.
\end{lemma}

\begin{proof}
We proceed by induction on the cardinality of $\overline{\Lambda}$.
If $|\overline{\Lambda}|=1$, then all $\Delta(\lambda)$ are projective
and all $\overline{\nabla}(\mu)$ are simple, so the claim is trivial.

Assume now that $|\overline{\Lambda}|>1$. Let $\lambda\in\Lambda$ be 
maximal. Then for all $\mu\not\in \overline{\lambda}$, the claim follows
from the inductive assumption applied to the stratified algebra
$B_{\overline{\lambda}}$. 

Assume, finally, that $\mu\in \overline{\lambda}$ is such that
$\overline{\nabla}(\mu)\not\in  A^{\downarrow}\text{-}\mathrm{gmod}$.
Then there exists $\nu\in\Lambda$ and an infinite sequence 
$0<j_1<j_2<\dots $ of positive integers such that for any $l\in\mathbb{N}$
there exists a nonzero homomorphism  from $P(\nu)\langle j_l\rangle$
to $\overline{\nabla}(\mu)$. Let $M_l$ denote the image of this 
homomorphism. Then $M_l$ has simple top $L(\nu)\langle j_l\rangle$
and simple socle $L(\mu)$ and all other composition subquotients of
the form $L(\nu')\langle j\rangle$, where $\nu'\prec \mu$ and
$1\leq j\leq j_l-1$.

The module $M_l\langle -j_l\rangle$ is thus a quotient of $P(\nu)$. Then
the socle $L(\mu)\langle -j_l\rangle$ of $M_l\langle -j_l\rangle$ gives
rise to a nonzero homomorphism from $P(\mu)\langle -j_l\rangle$ to
$P(\nu)$. Since $\mu$ is maximal and all other composition subquotients of 
$M_l\langle -j_l\rangle$ are of the form $L(\nu')\langle j\rangle$ for
some $\nu'\prec \mu$, the above homomorphism gives rise to an
occurrence of the standard module $\Delta(\mu)\langle -j_l\rangle$
in the standard filtration of $P(\nu)$. However, we have infinitely many
$j_l$'s and, at the same time, the standard filtration of $P(\nu)$ is
finite. This is a contradiction, which yields the claim of the lemma.
\end{proof}

\begin{lemma}\label{lem5}
For all $i,j\in\mathbb{Z}$ such that $i\geq 0$, and all 
$\lambda,\mu\in\Lambda$ we have
\begin{displaymath}
\mathrm{ext}^i_A(\Delta(\lambda),\overline{\nabla}(\mu)\langle j\rangle)=
\begin{cases}
\Bbbk, & i=j=0,\lambda=\mu; \\
0, & \text{otherwise}.
\end{cases}
\end{displaymath}
\end{lemma}

\begin{proof}
We proceed by induction on the cardinality of $\overline{\Lambda}$.
If $|\overline{\Lambda}|=1$, then all $\Delta(\lambda)$ are projective
and all $\overline{\nabla}(\mu)$ are simple, so the claim is trivial.

Assume now that $|\overline{\Lambda}|>1$. Let $\lambda'\in\Lambda$ be 
maximal. Then, by definitions, the module $\Delta(\lambda)$ is projective 
for all $\lambda\in \overline{\lambda'}$. Hence for such $\lambda$ the 
claim of the lemma follows from the definition of $\overline{\nabla}(\mu)$. 
If $\lambda,\mu\not\in \overline{\lambda'}$, the claim follows from the
inductive assumption applied to the standardly stratified algebra 
$B_{\overline{\lambda'}}$ and Lemma~\ref{lem3}. 

Consider now the case when $\mu\in \overline{\lambda'}$ and $\lambda\not\in
\overline{\lambda'}$. Then $\Delta(\lambda)$ does not have any composition
subquotient of the form $L(\mu)\langle j\rangle$ and hence
\begin{displaymath}
\mathrm{hom}_A(\Delta(\lambda),\overline{\nabla}(\mu)\langle j\rangle)=0.
\end{displaymath}
Let us check that
\begin{equation}\label{eq1}
\mathrm{ext}^1_A(\Delta(\lambda),\overline{\nabla}(\mu)\langle j\rangle)=0
\end{equation}
for all $j$. Applying $\mathrm{hom}_A(\Delta(\lambda),{}_-)$ to the
short exact sequence 
\begin{displaymath}
\overline{\nabla}(\mu)\langle j\rangle\hookrightarrow
I(\mu)\langle j\rangle\tto \mathrm{Coker},
\end{displaymath}
we obtain the exact sequence
\begin{displaymath}
\mathrm{hom}_A(\Delta(\lambda),\mathrm{Coker})\to
\mathrm{ext}^1_A(\Delta(\lambda),\overline{\nabla}(\mu)\langle j\rangle)\to 
\mathrm{ext}^1_A(\Delta(\lambda),I(\mu)\langle j\rangle).
\end{displaymath}
Here the right term equals zero by the injectivity of $I(\mu)$. 
By the definition of $\overline{\nabla}(\mu)$, the socle of 
$\mathrm{Coker}$ has (up to shift) only simple modules of the form 
$L(\nu)$, where $\nu\in \overline{\lambda'}$, which implies that 
the left term equals zero as well.
The equality \eqref{eq1} follows.

Now we prove our claim by induction on $\lambda$ with respect to 
$\preceq$ (as mentioned above, the claim is true for $\lambda$ maximal).
Apply $\mathrm{hom}_A({}_-,\overline{\nabla}(\mu)\langle j\rangle)$ to the
short exact sequence
\begin{equation}\label{eq3}
\mathrm{Ker}\hookrightarrow P(\lambda)\tto \Delta(\lambda)
\end{equation}
and, using the projectivity of $P(\lambda)$, for each $i>1$ obtain 
the following exact sequence:
\begin{displaymath}
0\to\mathrm{ext}^{i-1}_A(\mathrm{Ker},\overline{\nabla}(\mu)\langle j\rangle)
\to\mathrm{ext}^{i}_A(\Delta(\lambda),\overline{\nabla}(\mu)\langle j\rangle)
\to0.
\end{displaymath}
Since $A$ is standardly stratified, $\mathrm{Ker}$ has a finite filtration
by standard modules of the form $\Delta(\nu)$, where $\lambda\prec\nu$,
(up to shift). Hence, from the inductive assumption we get  
$\mathrm{ext}^{i-1}_A(\mathrm{Ker},\overline{\nabla}(\mu)\langle j\rangle)=0$. 
This yields $\mathrm{ext}^{i}_A(\Delta(\lambda),
\overline{\nabla}(\mu)\langle j\rangle)=0$ and completes the proof.
\end{proof}

\begin{corollary}\label{corlem5}
Let $A$ be a positively graded standardly stratified algebra.
\begin{enumerate}[(i)]
\item\label{corlem5.1}
For any $M\in\mathcal{F}^{\downarrow}(\Delta)$,
$\lambda\in\Lambda$, $i\in\mathbb{N}$ and $j\in\mathbb{Z}$ we have
$\mathrm{ext}^i(M,\overline{\nabla}(\lambda)\langle j\rangle)=0$.
\item\label{corlem5.2}
For any $M\in\mathcal{F}^{\downarrow}(\overline{\nabla})$,
$\lambda\in\Lambda$, $i\in\mathbb{N}$ and $j\in\mathbb{Z}$ we have
$\mathrm{ext}^i(\Delta(\lambda)\langle j\rangle,M)=0$.
\end{enumerate}
\end{corollary}

\begin{proof}
It is certainly enough to prove statement \eqref{corlem5.1} 
in the case when $M$ has a filtration of the form \eqref{eqfiltration}.
As $\lim_{i\rightarrow+\infty}\mathfrak{b}
(M^{(i)})=+\infty$ and $\overline{\nabla}(\lambda)$ is finite dimensional
(Lemma~\ref{lem4}),
there exists $n\in\mathbb{Z}$ such that for any $i\in\mathbb{Z}$
with $\overline{\nabla}(\lambda)\langle j\rangle_i\neq 0$ we
have $i<\mathfrak{b}(M^{(n)})$. Since $A$ is positively graded,
there are no homomorphisms from any component of the projective 
resolution of $M^{(n)}$ to $\overline{\nabla}(\lambda)\langle j\rangle$.
This means that all extentions from $M^{(n)}$ to 
$\overline{\nabla}(\lambda)\langle j\rangle$ vanish. At the same time,
the quotient $M/M^{(n)}$ has a finite filtration by standard modules
and hence all extensions from $M/M^{(n)}$ to 
$\overline{\nabla}(\lambda)\langle j\rangle$  vanish by
Lemma~\ref{lem5}. Statement \eqref{corlem5.1} follows.

It is certainly enough to prove statement \eqref{corlem5.2} 
in the case when $M$ has a filtration of the form \eqref{eqfiltration}
(with subquotients being proper costandard modules). 
Let $\mathcal{P}^{\bullet}$ be the minimal projective resolution
of $\Delta(\lambda)\langle j\rangle$. As
every indecomposable projective has a finite standard filtration, 
it follows that $\mathcal{P}^{\bullet}$ has only finitely many
nonzero components, moreover, each of them is a finite direct sum of
projective modules. As $\lim_{i\rightarrow+\infty}\mathfrak{b}
(M^{(i)})=+\infty$, there exists $n\in\mathbb{N}$ such that
there are no maps from any $\mathcal{P}^{i}$ to $M^{(n)}$, in particular,
all extensions from $\Delta(\lambda)\langle j\rangle$
to $M^{(n)}$ vanish. At the same time, the quotient $M/M^{(n)}$ has a 
finite filtration by proper costandard modules and hence all extensions 
from $\Delta(\lambda)\langle j\rangle$ to 
$M/M^{(n)}$ vanish by Lemma~\ref{lem5}.
Statement \eqref{corlem5.2} follows and the proof is complete.
\end{proof}

The following lemma is just an observation that the category
$\mathcal{F}^{\downarrow}(\overline{\nabla})$ can, in fact, be
defined in a somewhat easier way than the one we used. For the category
$\mathcal{F}^{\downarrow}(\Delta)$ this is not possible in the 
general case, see Example~\ref{exm2}.

\begin{lemma}\label{lemlemnew}
Any module from $\mathcal{F}^{\downarrow}(\overline{\nabla})$ has a 
filtration of the form \eqref{eqfiltration}.
\end{lemma}

\begin{proof}
Let $X,Z\in \mathcal{C}$ and
\begin{displaymath}
X=X^{(0)}\supseteq X^{(1)}\supseteq X^{(2)}\supseteq\dots,
\end{displaymath}
and
\begin{displaymath}
Z=Z^{(0)}\supseteq Z^{(1)}\supseteq Z^{(2)}\supseteq\dots,
\end{displaymath}
be filtrations of the form \eqref{eqfiltration}. Assume that 
$Y\in A^{\downarrow}\text{-}\mathrm{gmod}$ is such that there is
a short exact sequence 
\begin{displaymath}
0\to X\to Y\to Z\to 0. 
\end{displaymath}
To prove the claim of the lemma it is enough to show that 
$Y$ has a filtration of the form \eqref{eqfiltration}.

Since all costandard modules
are finite dimensional (Lemma~\ref{lem4}) and
$\displaystyle\lim_{i\rightarrow+\infty}\mathfrak{b}(Z^{(i)})=+\infty$,
there exists $k\in \{0,1,2,\dots\}$ such that for any $i\in\mathbb{Z}$
with $(X^{(0)}/X^{(1)})_i\neq 0$ we have $i<\mathbf{b}(Z^{(k)})$.

Now for $i=0,1,\dots,k$, we let $Y^{(i)}$ be the full preimage of
$Z^{(i)}$ in $Y$ under the projection $Y\tto Z$. 
In this way we get the first part of the 
filtration of $Y$ with proper costandard subquotients. On the
next step we let $Y^{(k+1)}$ denote the submodule of 
$Y^{(k)}$ generated by $X^{(1)}$ and $Y^{(k)}_i$, where
$i\geq \mathbf{b}(Z^{(k)})$.
Then $Y^{(k+1)}+X^{(0)}=Y^{(k)}$ by construction. At the same time,
from our choice of $k$ in the previous paragraph it follows that 
$Y^{(k+1)}\cap X^{(0)}=X^{(1)}$ and hence
\begin{displaymath}
Y^{(k)}/Y^{(k+1)}\cong  X^{(0)}/X^{(1)},
\end{displaymath}
which is a proper costandard module. 

Now we proceed in the same way constructing a proper costandard 
filtration for $Y^{(k+1)}$. The condition 
$\displaystyle\lim_{i\rightarrow+\infty}\mathfrak{b}(Y^{(i)})=+\infty$
follows from the construction. This completes the proof.
\end{proof}

\begin{lemma}\label{lem6}
Let $M\in A^{\downarrow}\text{-}\mathrm{gmod}$ be such that 
$\mathrm{ext}^{1}_A(\Delta(\lambda)\langle j\rangle,M)=0$ for all
$\lambda$ and $j$. Then $M\in \mathcal{F}^{\downarrow}(\overline{\nabla})$.
\end{lemma}

\begin{proof}
First let us show that the conditions of the lemma imply
\begin{equation}\label{eq2}
\mathrm{ext}^{i}_A(\Delta(\lambda)\langle j\rangle,M)=0
\end{equation}
for all $j$, all $\lambda$ and all $i>0$. If $\lambda$ is maximal, then
the corresponding $\Delta(\lambda)$ is projective and the claim is clear.
Otherwise, we proceed by induction with respect to the preorder $\preceq$.
We apply $\mathrm{hom}_A({}_-,M)$ to the short exact sequence
\eqref{eq3} and the equality \eqref{eq2} follows from the inductive 
assumption by the dimension shift in the obtained long exact sequence.

We proceed by induction on the cardinality of $\overline{\Lambda}$.
If $|\overline{\Lambda}|=1$, then 
$\mathcal{F}^{\downarrow}(\overline{\nabla})=
A^{\downarrow}\text{-}\mathrm{gmod}$
and the claim is trivial.

Assume now that $|\overline{\Lambda}|>1$ and let $\lambda'\in\Lambda$
be maximal. Let $N$ denote the maximal submodule of $M$, which does
not contain any composition factors of the form $L(\mu)$, where
$\mu\in \overline{\lambda'}$ (up to shift). Let
$\nu\not\in \overline{\lambda'}$. Applying 
$\mathrm{hom}_A(\Delta(\nu)\langle j\rangle,{}_-)$ to the short 
exact sequence
\begin{equation}\label{eq4}
N\hookrightarrow M\tto\mathrm{Coker},
\end{equation}
we obtain the exact sequence
\begin{displaymath}
\mathrm{hom}_A(\Delta(\nu)\langle j\rangle,\mathrm{Coker})\to
\mathrm{ext}^1_A(\Delta(\nu)\langle j\rangle,N)\to
\mathrm{ext}^1_A(\Delta(\nu)\langle j\rangle,M).
\end{displaymath}
Here the right term is zero by our assumptions and the left term is zero
by the definition of $N$. This implies that the middle term is zero,
which yields $\mathrm{ext}^1_{B_{\overline{\lambda'}}}
(\Delta(\nu)\langle j\rangle,N)=0$
by Lemma~\ref{lem3}. Applying the inductive assumption to the
standardly stratified algebra $B_{\overline{\lambda'}}$, we obtain
that $N\in \mathcal{F}^{\downarrow}(\overline{\nabla})$. 

Since $\mathcal{F}^{\downarrow}(\overline{\nabla})$ is extension closed,
to complete the proof we are left to  show that 
$\mathrm{Coker}\in \mathcal{F}^{\downarrow}(\overline{\nabla})$.
Applying $\mathrm{hom}_A(\Delta(\lambda)\langle j\rangle,{}_-)$ to \eqref{eq4}
and using \eqref{eq2}, the previous paragraph and Lemma~\ref{lem5},
we obtain that 
\begin{equation}\label{eq5}
\mathrm{ext}^{i}_A(\Delta(\lambda)\langle j\rangle,\mathrm{Coker})=0
\end{equation}
for all $j$, $\lambda$ and $i>0$.

If $\mathrm{Coker}=0$, we are done. Otherwise, there exists some
$\mu\in\overline{\lambda'}$ and a maximal possible  $j'\in\mathbb{Z}$ 
such that there is a nonzero homomorphism from $\mathrm{Coker}$ to 
$I(\mu)\langle j'\rangle$. Let $K$ denote the image of this 
homomorphism. Applying 
$\mathrm{hom}_A(\Delta(\lambda)\langle j\rangle,{}_-)$ to the 
short exact sequence
\begin{equation}\label{eq6}
\mathrm{Ker}\hookrightarrow \mathrm{Coker}\tto K,
\end{equation}
and using the definition of $K$, we obtain that 
\begin{equation}\label{eq7}
\mathrm{ext}^{1}_A(\Delta(\lambda)\langle j\rangle,\mathrm{Ker})=0
\end{equation}
for all $\lambda$ and $j$. The equality \eqref{eq7},
the corresponding equalities \eqref{eq2} (for $M=\mathrm{Ker}$) 
and the dimension shift with respect to \eqref{eq6} then imply
\begin{equation}\label{eq8}
\mathrm{ext}^{1}_A(\Delta(\lambda)\langle j\rangle,K)=0
\end{equation}
for all $\lambda$ and $j$.

By the definition of $K$ we have a short exact sequence
\begin{equation}\label{eq9}
K\hookrightarrow \overline{\nabla}(\mu)\langle j'\rangle\tto C'
\end{equation}
for some cokernel $C'$. By the definition of $\overline{\nabla}(\mu)$, 
all composition subquotients of $C'$ have the form $L(\nu)$, 
where $\nu\prec \mu$ (up to shift).
Let $\lambda\in\Lambda$ be such that $\lambda\prec \mu$. Applying 
$\mathrm{hom}_A(\Delta(\lambda)\langle j\rangle,{}_-)$  
to \eqref{eq9} we get the exact  sequence
\begin{equation}\label{eq10}
\mathrm{hom}_A(\Delta(\lambda)\langle j\rangle,
\overline{\nabla}(\mu)\langle j'\rangle)
\to \mathrm{hom}_A(\Delta(\lambda)\langle j\rangle,C')\to
\mathrm{ext}^1_A(\Delta(\lambda)\langle j\rangle,K).
\end{equation}
Here the left term is zero by the definition of $\overline{\nabla}(\mu)$
and the right hand term is zero by \eqref{eq8}. This yields that the
middle term is zero as well and thus $C'=0$, that is 
$K$ is a proper costandard module. 

We can now apply the same arguments as above to the module 
$\mathrm{Ker}$ in place of $\mathrm{Coker}$ and get the short 
exact sequence
\begin{displaymath}
\mathrm{Ker}'\hookrightarrow\mathrm{Ker} \tto K',
\end{displaymath}
where $K'$ is proper costandard and 
$\mathrm{ext}^{1}_A(\Delta(\lambda)\langle j\rangle,
\mathrm{Ker}')=0$ for all $\lambda$ and $j$. Proceeding 
inductively we obtain a (possibly infinite) decreasing filtration 
\begin{displaymath}
\mathrm{Coker}\supseteq
\mathrm{Ker}\supseteq\mathrm{Ker}'\supseteq\dots
\end{displaymath}
with proper costandard subquotients.
That $\displaystyle\lim_{i\rightarrow+\infty}\mathfrak{b}
(\mathrm{Coker}^{(i)})=+\infty$ follows from the construction
since all our modules are from $A^{\downarrow}\text{-}\mathrm{gmod}$,
all proper costandard modules (subquotients of the filtration of
$\mathrm{Coker}$) are finite-dimensional by Lemma~\ref{lem4},
and there are only finitely many proper costandard modules up to
isomorphism and shift (which implies that dimensions of proper costandard
modules are uniformly bounded). Therefore we get $\mathrm{Coker}\in  
\mathcal{F}^{\downarrow}(\overline{\nabla})$. The claim of the
lemma follows.
\end{proof}

\begin{lemma}\label{lem7}
Let $M\in A^{\downarrow}\text{-}\mathrm{gmod}$ be such that 
$\mathrm{ext}^{1}_A(M,\overline{\nabla}(\mu)\langle j\rangle)=0$ for all
$\mu$ and $j$. Then $M\in \mathcal{F}^{\downarrow}(\Delta)$.
\end{lemma}

\begin{proof}
Let $M\in A^{\downarrow}\text{-}\mathrm{gmod}$ be such that 
$\mathrm{ext}^{1}_A(M,\overline{\nabla}(\mu)\langle j\rangle)=0$
for all $\mu$ and $j$. We again proceed by induction on 
$|\overline{\Lambda}|$. If $|\overline{\Lambda}|=1$, then 
proper costandard modules are simple and hence $M$ is projective.
All indecomposable projective modules belong
to $\mathcal{F}^{\downarrow}(\Delta)$ as $A$ is standardly stratified.
Using this it is easy to check that all projective modules in
$A^{\downarrow}\text{-}\mathrm{gmod}$ belong to $\mathcal{F}^{\downarrow}(\Delta)$. So, in the case 
$|\overline{\Lambda}|=1$ the claim of the lemma is true.

If $|\overline{\Lambda}|>1$, we take some maximal $\nu\in\Lambda$
and denote by $N$ the sum of all images of all possible homomorphisms from
$\Delta(\lambda)\langle j\rangle$, where $\lambda\in\overline{\nu}$ and 
$j\in\mathbb{Z}$, to $M$. Then we have a short exact sequence
\begin{equation}\label{eqnew5}
N\hookrightarrow M\tto \mathrm{Coker}.
\end{equation}
Compare with \eqref{eq4} in the proof of Lemma~\ref{lem6}.
Using arguments similar to those in the latter proof, one shows
that $\mathrm{ext}^1_A(\mathrm{Coker},\overline{\nabla}(\mu)\langle 
j\rangle)=0$ for all $\mu\in \Lambda\setminus\overline{\nu}$ and all
$j$. By construction we have that $\mathrm{Coker}$ is in fact a
$B_{\overline{\nu}}\,$-module. Therefore, 
using Lemma~\ref{lem3} and the inductive assumption 
we get $\mathrm{Coker}\in\mathcal{F}^{\downarrow}(\Delta)$.
From Corollary~\ref{corlem5}\eqref{corlem5.1} we thus get
\begin{equation}\label{eqnn5}
\mathrm{ext}^i_A(\mathrm{Coker},\overline{\nabla}(\mu)
\langle j\rangle)=0 
\end{equation}
for all $\mu\in \Lambda$, $j\in\mathbb{Z}$ and $i\in\mathbb{N}$.

Furthermore, for any $\mu$ and $j$ we also 
have the following part of the 
long exact  sequence associated with \eqref{eqnew5}:
\begin{displaymath}
\mathrm{ext}^1_A(M,\overline{\nabla}(\mu)\langle 
j\rangle) \to
\mathrm{ext}^1_A(N,\overline{\nabla}(\mu)\langle 
j\rangle) \to
\mathrm{ext}^2_A(\mathrm{Coker},\overline{\nabla}(\mu)\langle 
j\rangle).
\end{displaymath}
The left term is zero by our assumptions and the right term is zero
by \eqref{eqnn5}. Therefore for all $\mu$ and $j$ we have
\begin{equation}\label{eq567}
\mathrm{ext}^1_A(N,\overline{\nabla}(\mu)\langle j\rangle)=0.
\end{equation}
Fix now $\mu\in\Lambda$ and $j\in\mathbb{Z}$ and denote by 
$C$ the cokernel of the natural inclusion 
$L(\mu)\langle j\rangle\hookrightarrow 
\overline{\nabla}(\mu)\langle j\rangle$.
Applying $\mathrm{hom}_A(N,{}_-)$ to the short exact sequence
\begin{displaymath}
L(\mu)\langle j\rangle\hookrightarrow
\overline{\nabla}(\mu)\langle j\rangle\tto C, 
\end{displaymath}
and using \eqref{eq567} and the fact 
that $\mathrm{hom}_A(N,C)=0$ by construction, we obtain that 
$\mathrm{ext}^1_A(N,L(\mu)\langle j\rangle)=0$
for any $\mu$ and $j$. This yields that $N$ is projective and thus
belongs to $\mathcal{F}^{\downarrow}(\Delta)$.   Since 
$\mathcal{F}^{\downarrow}(\Delta)$ is closed under extensions,
the claim of the lemma follows.
\end{proof}

\begin{proof}[Proof of Theorem~\ref{thm2}.]
Let 
{\small
\begin{displaymath}
\begin{array}{rcl} 
\mathcal{X}&=&
\{M\in A^{\downarrow}\text{-}\mathrm{gmod}\,:\,
\mathrm{ext}^i_A(M,\overline{\nabla}(\lambda)\langle j\rangle)=0,
\forall j\in\mathbb{Z},i>0,\lambda\in\Lambda\};\\
\mathcal{Y}&=&
\{M\in A^{\downarrow}\text{-}\mathrm{gmod}\,:\,
\mathrm{ext}^1_A(M,\overline{\nabla}(\lambda)\langle j\rangle)=0,
\forall j\in\mathbb{Z},\lambda\in\Lambda\}.
\end{array}
\end{displaymath}
}
The inclusion $\mathcal{X}\subseteq \mathcal{Y}$ is obvious. The inclusion
$\mathcal{Y}\subseteq \mathcal{F}^{\downarrow}(\Delta)$ follows from
Lemma~\ref{lem7}. The inclusion
$\mathcal{F}^{\downarrow}(\Delta)\subseteq \mathcal{X}$ 
follows from Corollary~\ref{corlem5}\eqref{corlem5.1}.
This proves Theorem~\ref{thm2}\eqref{thm2-1}.
Theorem~\ref{thm2}\eqref{thm2-2} is proved similarly
using Lemma~\ref{lem6} instead of Lemma~\ref{lem7}
and Corollary~\ref{corlem5}\eqref{corlem5.2} instead of 
Corollary~\ref{corlem5}\eqref{corlem5.1}.
\end{proof}

\begin{corollary}\label{cor11}
Let $A$ be a positively graded standardly stratified algebra.
\begin{enumerate}[(i)]
\item \label{cor11-1}
For every $M\in \mathcal{F}^{\downarrow}(\Delta)$,
$\lambda\in\Lambda$ and $j\in\mathbb{Z}$ the multiplicity
of $\Delta(\lambda)\langle j\rangle$ in any standard filtration of 
$M$ is well-defined, finite and equals
$\dim\mathrm{hom}_A(M,\overline{\nabla}(\lambda)\langle j\rangle)$.
\item \label{cor11-2}
For every $M\in \mathcal{F}^{\downarrow}(\overline{\nabla})$,
$\lambda\in\Lambda$ and $j\in\mathbb{Z}$ the multiplicity
of $\overline{\nabla}(\lambda)\langle j\rangle$ in any proper 
costandard filtration of  $M$ is well-defined, finite and equals
$\dim\mathrm{hom}_A(\Delta(\lambda)\langle j\rangle,M)$.
\end{enumerate}
\end{corollary}

\begin{proof}
Follows from Lemma~\ref{lem5} by standard arguments (see e.g.  \cite{Ri}).
\end{proof}

\begin{remark}\label{rem12}
{\rm
Note that the ungraded multiplicity of $\Delta(\lambda)$
(or $\overline{\nabla}(\lambda)$) in $M$ might be infinite.
}
\end{remark}

Let $\mathcal{F}^{\uparrow}(\overline{\nabla})$ denote the full subcategory 
of the category $A^{\uparrow}\text{-}\mathrm{gmod}$, which consists
of all modules $M$ admitting a (possibly infinite) filtration 
\begin{equation}\label{eq312}
0=M^{(0)}\subseteq M^{(1)}\subseteq M^{(2)}\subseteq\dots 
\end{equation}
such that $\displaystyle M=\bigcup_{i\geq 0}M^{(i)}$ and for every 
$i=0,1,\dots$ the subquotient  $M^{(i+1)}/M^{(i)}$ is isomorphic 
(up to shift) to some proper costandard module. Since all
proper costandard modules are finite dimensional (Lemma~\ref{lem4})
from the dual version of Lemma~\ref{lemlemnew} one obtains that
$\mathcal{F}^{\uparrow}(\overline{\nabla})$ is closed under finite 
extensions.

\begin{theorem}\label{thm2new}
We have
{\small
\begin{displaymath}
\begin{array}{rcl} 
\mathcal{F}^{\uparrow}(\overline{\nabla})&=&
\{M\in A^{\uparrow}\text{-}\mathrm{gmod}\,:\,
\mathrm{ext}^i_A(\Delta(\lambda)\langle j\rangle,M)=0,
\forall j\in\mathbb{Z},i>0,\lambda\in\Lambda\}\\
&=&
\{M\in A^{\uparrow}\text{-}\mathrm{gmod}\,:\,
\mathrm{ext}^1_A(\Delta(\lambda)\langle j\rangle,M)=0,
\forall j\in\mathbb{Z},\lambda\in\Lambda\}.
\end{array}
\end{displaymath}
}
\end{theorem}

\begin{proof}
Set
\begin{gather*}
\mathcal{X}=\{M\in A^{\uparrow}\text{-}\mathrm{gmod}\,:\,
\mathrm{ext}^1_A(\Delta(\lambda)\langle j\rangle,M)=0,
\forall j\in\mathbb{Z},\lambda\in\Lambda\},\\
\mathcal{Y}=\{M\in A^{\uparrow}\text{-}\mathrm{gmod}\,:\,
\mathrm{ext}^i_A(\Delta(\lambda)\langle j\rangle,M)=0,
\forall j\in\mathbb{Z},i>0,\lambda\in\Lambda\}.
\end{gather*}
Obviously, $\mathcal{Y}\subseteq \mathcal{X}$.

Let $M\in \mathcal{F}^{\uparrow}(\overline{\nabla})$,
$\lambda\in\Lambda$ and $j\in\mathbb{Z}$. Assume that \eqref{eq312} 
gives a proper costandard filtration of $M$. As $M\in 
A^{\uparrow}\text{-}\mathrm{gmod}$ and $\Delta(\lambda)\in
A^{\downarrow}\text{-}\mathrm{gmod}$, it follows that there exists 
$k\in\mathbb{N}$ such that 
\begin{displaymath}
\mathrm{ext}^i_A(\Delta(\lambda)\langle j\rangle,M/M^{(k)})=0 
\end{displaymath}
for all $i\geq 0$. At the same time we have 
\begin{displaymath}
\mathrm{ext}^i_A(\Delta(\lambda)\langle j\rangle,M^{(k)})=0 
\end{displaymath}
for all $i>0$ by Lemma~\ref{lem5}. Hence 
\begin{displaymath}
\mathrm{ext}^i_A(\Delta(\lambda)\langle j\rangle,M)=0 
\end{displaymath}
for all $i>0$ and thus 
$\mathcal{F}^{\uparrow}(\overline{\nabla})\subseteq \mathcal{Y}$.

It is left to show that $\mathcal{X}\subseteq 
\mathcal{F}^{\uparrow}(\overline{\nabla})$. We will do this by induction
on $|\overline{\Lambda}|$. If $|\overline{\Lambda}|=1$, then all
proper standard modules are simple, which yields
$\mathcal{F}^{\uparrow}(\overline{\nabla})=
A^{\uparrow}\text{-}\mathrm{gmod}$. In this case
the inclusion $\mathcal{X}\subseteq 
\mathcal{F}^{\uparrow}(\overline{\nabla})$ is obvious.

If $|\overline{\Lambda}|>1$ we fix some maximal $\mu\in\Lambda$.
Let $M\in \mathcal{X}$. Denote by $N$ the maximal submodule of
$M$ satisfying $[N:L(\nu)\langle j\rangle]=0$ for all
$\nu\in\overline{\mu}$ and $j\in\mathbb{Z}$. For
$\lambda\in\Lambda$ and $j\in\mathbb{Z}$, applying the functor  
$\mathrm{hom}_A(\Delta(\lambda)\langle j\rangle,{}_-)$ to 
the short exact sequence
\begin{displaymath}
N\hookrightarrow M\tto \mathrm{Coker},
\end{displaymath}
and using $M\in \mathcal{X}$, gives the following exact sequences:
\begin{equation}\label{eqn2-2}
\mathrm{hom}_A(\Delta(\lambda)\langle j\rangle,\mathrm{Coker})
\to \mathrm{ext}^1_A(\Delta(\lambda)\langle j\rangle,N)
\to 0
\end{equation}
and
\begin{equation}\label{eqn3-2}
0\to \mathrm{ext}^1_A(\Delta(\lambda)\langle j\rangle,\mathrm{Coker})
\to \mathrm{ext}^2_A(\Delta(\lambda)\langle j\rangle,N).
\end{equation}
By construction, any simple subquotient in the socle of 
$\mathrm{Coker}$ has the form $L(\nu)\langle j\rangle$
for some $\nu\in\overline{\mu}$ and $j\in\mathbb{Z}$.
Therefore, since $\mu$ is maximal, in the case 
$\lambda\not\in\overline{\mu}$ we have 
$\mathrm{hom}_A(\Delta(\lambda)\langle j\rangle,\mathrm{Coker})=0$
and hence $\mathrm{ext}^1_A(\Delta(\lambda)\langle j\rangle,N)=0$
from \eqref{eqn2-2}. For $\lambda\in\overline{\mu}$ the module
$\Delta(\lambda)\langle j\rangle$ is projective and hence
$\mathrm{ext}^1_A(\Delta(\lambda)\langle j\rangle,N)=0$ as well.
This implies $N\in \mathcal{X}$. As, by construction,
$N\in B_{\overline{\mu}}\,\text{-}\mathrm{mod}$,
using Lemma~\ref{lem3} and the inductive assumption
we obtain $N\in \mathcal{F}^{\uparrow}(\overline{\nabla})$.
As the inclusion $\mathcal{F}^{\uparrow}(\overline{\nabla})\subseteq
\mathcal{Y}$ is already proved, we have $N\in \mathcal{Y}$ and 
from \eqref{eqn3-2} it follows that $\mathrm{Coker}\in\mathcal{X}$.

Since $\mathcal{F}^{\uparrow}(\overline{\nabla})$ is closed under finite
extensions, it is left to show that $\mathrm{Coker}\in
\mathcal{F}^{\uparrow}(\overline{\nabla})$. If
$\mathrm{Coker}=0$, we have nothing to do.
If $\mathrm{Coker}\neq 0$, we choose maximal $k\in \mathbb{Z}$ such 
that $\mathrm{Coker}_k\neq 0$. Denote by $V$ the intersection
of the kernels of all possible maps from $\mathrm{Coker}$ to 
$I(\nu)\langle j\rangle$, where 
$\nu\in\overline{\mu}$ and $-j<k$, and consider the corresponding
short exact sequence
\begin{equation}\label{eqn1}
V\hookrightarrow  \mathrm{Coker}\tto \mathrm{Coker}'.
\end{equation}
From the construction it follows that the socle of $V$ is $V_k$
and that for any $j<k$ every composition subquotient of 
$V_j$ has the form $L(\nu)\langle -j\rangle$ for some
$\nu\not\in\overline{\mu}$. Therefore, taking the injective envelope
of $V$ and using the definition of proper standard modules, we
obtain that $V$ is a submodule of a finite direct sum of proper standard
modules (such that the socles of $V$ and of this direct sum 
agree). In particular, $V$ is finite dimensional
as both $V_k$ and all proper standard modules are (Lemma~\ref{lem4}). 
Hence $V\in A^{\downarrow}\text{-}\mathrm{gmod}$. 

For $\lambda\in \Lambda$ and $j\in \mathbb{Z}$, applying the functor  
$\mathrm{hom}_A(\Delta(\lambda)\langle j\rangle,{}_-)$ to \eqref{eqn1}
and using $\mathrm{Coker}\in \mathcal{X}$ gives the following exact sequences:
\begin{equation}\label{eqn2}
\mathrm{hom}_A(\Delta(\lambda)\langle j\rangle,\mathrm{Coker}')
\to \mathrm{ext}^1_A(\Delta(\lambda)\langle j\rangle,V)
\to 0
\end{equation}
and
\begin{equation}\label{eqn3}
0\to \mathrm{ext}^1_A(\Delta(\lambda)\langle j\rangle,\mathrm{Coker}')
\to \mathrm{ext}^2_A(\Delta(\lambda)\langle j\rangle,V).
\end{equation}
If $\lambda\not\in\overline{\mu}$, then, by the definition of the module
$\mathrm{Coker}'$, we have 
$\mathrm{hom}_A(\Delta(\lambda)\langle j\rangle,\mathrm{Coker}')=0$
and hence $\mathrm{ext}^1_A(\Delta(\lambda)\langle j\rangle,V)=0$
from \eqref{eqn2}.
If $\lambda\in\overline{\mu}$, then
$\Delta(\lambda)\langle j\rangle$ is projective by the maximality of $\mu$
and $\mathrm{ext}^1_A(\Delta(\lambda)\langle j\rangle,V)=0$ automatically.
Hence $V\in\mathcal{X}$. Since $V\in A^{\downarrow}\text{-}\mathrm{gmod}$
as shown above, from Theorem~\ref{thm2}\eqref{thm2-2} we deduce that
$V$ has a (finite) proper standard filtration and thus
$V\in \mathcal{F}^{\uparrow}(\overline{\nabla})$.
Using the already proved inclusion 
$\mathcal{F}^{\uparrow}(\overline{\nabla})\subseteq \mathcal{Y}$ and
\eqref{eqn3} we also get $\mathrm{Coker}'\in\mathcal{X}$.
Note that $\mathrm{Coker}'_k=0$ by construction.

Applying now the same arguments to $\mathrm{Coker}'$ and proceeding
inductively (decreasing $k$) we construct a (possibly infinite)
proper costandard filtration of $\mathrm{Coker}'$
of the form \eqref{eq312}. This claim of the theorem follows.
\end{proof}

The following claim is a weak version of \cite[Lemma~2.1]{Dl} and 
\cite[Theorem~1]{Fr}. The original statement also contains the 
converse assertion that the fact that indecomposable injective 
$A$-modules belong to $\mathcal{F}^{\uparrow}(\overline{\nabla})$ 
guarantees that $A$ is standardly stratified.

\begin{corollary}[Weak Dlab's theorem]\label{cor2new}
All indecomposable injective $A$-modules belong to
$\mathcal{F}^{\uparrow}(\overline{\nabla})$.
\end{corollary}

\begin{proof}
If $I$ in an indecomposable injective  $A$-module, then we obviously have
$\mathrm{ext}^i_A(\Delta(\lambda)\langle j\rangle,I)=0$ for all
$j\in\mathbb{Z}$, $i>0$ and $\lambda\in\Lambda$, so the claim follows
from Theorem~\ref{thm2new}.
\end{proof}

The following statement generalizes the corresponding results of
\cite{Ri,AHLU,Fr}:

\begin{theorem}[Construction of tilting modules]
\label{thm8} 
Let $A$ be a positively graded standardly stratified algebra.
\begin{enumerate}[(i)]
\item \label{thm8-1} The category  
$\mathcal{F}^{\downarrow}(\Delta)\cap 
\mathcal{F}^{\downarrow}(\overline{\nabla})$ is closed with
respect to taking direct sums and direct summands.
\item \label{thm8-2} For every $\lambda\in\Lambda$ there is
a unique indecomposable object $T(\lambda)\in 
\mathcal{F}^{\downarrow}(\Delta)\cap 
\mathcal{F}^{\downarrow}(\overline{\nabla})$ such that there is 
a short exact sequence
\begin{displaymath}
\Delta(\lambda)\hookrightarrow  T(\lambda)\tto
\mathrm{Coker},
\end{displaymath}
with $\mathrm{Coker}\in \mathcal{F}^{\downarrow}(\Delta)$. 
\item \label{thm8-3} 
Every indecomposable object in $\mathcal{F}^{\downarrow}(\Delta)\cap 
\mathcal{F}^{\downarrow}(\overline{\nabla})$ has the form
$T(\lambda)\langle j\rangle$ for some $\lambda\in \Lambda$ and
$j\in\mathbb{Z}$.
\end{enumerate}
\end{theorem}

We would need the following lemmata:

\begin{lemma}\label{lem9}
For all $\lambda,\mu\in\Lambda$, $i\geq 0$ and all $j\gg 0$ we have
\begin{displaymath}
\mathrm{ext}^i_A(\Delta(\lambda)\langle j\rangle,\Delta(\mu))=0.
\end{displaymath}
\end{lemma}

\begin{proof}
We proceed by induction with respect to $\preceq$. If $\lambda$ is
maximal, the module $\Delta(\lambda)$ is projective and the claim
is trivial for $i>0$. For $i=0$ the claim follows from the fact that
$A$ is positively graded. Now, if $\lambda$ is not maximal, we consider
the short exact sequence \eqref{eq3}. In this sequence $\mathrm{Ker}$
has a {\em finite} filtration by (shifted) standard modules, whose
indexes are strictly greater than $\lambda$ with respect to $\preceq$. Hence
the claim follows by the usual dimension shift (note that it 
is enough to consider only finitely many values of $i$, namely
$i\leq |\Lambda|$).
\end{proof}

\begin{lemma}\label{lem10}
For all $\lambda,\mu\in\Lambda$ and $j\in\mathbb{Z}$ the inequality
\begin{displaymath}
\mathrm{ext}^1_A(\Delta(\lambda)\langle j\rangle,\Delta(\mu))\neq 0. 
\end{displaymath}
implies $\lambda\prec\mu$.
\end{lemma}

\begin{proof}
If $\lambda\not\prec\mu$, then, using Lemma~\ref{lem3}, we may assume that
$\lambda$ is maximal. In this case $\Delta(\lambda)$ is projective
and the claim becomes trivial.
\end{proof}

\begin{lemma}\label{lemnn7}
For all $M\in\mathcal{F}^{\downarrow}(\Delta)$,
$N\in\mathcal{F}^{\downarrow}(\overline{\nabla})$ and $i\in\mathbb{N}$
we have $\mathrm{ext}^i_A(M,N)=0$.
\end{lemma}

\begin{proof}
It is enough to prove the claim in the case when $M$ has a filtration
of the form \eqref{eqfiltration}. Let $\lambda$ be a maximal index
occurring in standard subquotients of $M$. Then from Lemma~\ref{lem10}
we have that all corresponding standard subquotients do not extend
any other standard subquotients of $M$. Therefore $M$ has a
submodule isomorphic to a direct sum of shifted
$\Delta(\lambda)$ such that the cokernel has a standard filtration
in which no subquotient of the form  $\Delta(\lambda)$ (up to shift) occur.
Since $\Lambda$ is finite, proceeding inductively we construct
a finite filtration of $M$ whose subquotients
are direct sums of standard modules. This means that it is enough to
prove the claim in the case when $M$ is a direct sum of standard
modules. In this case the claim follows from 
Corollary~\ref{corlem5}\eqref{corlem5.2}.
\end{proof}

\begin{proof}[Proof of Theorem~\ref{thm8}.]
Statement \eqref{thm8-1} follows from the additivity of the conditions,
which appear on the right hand side in the formulae of Theorem~\ref{thm2}.

The existence part of statement \eqref{thm8-2} is proved using the 
usual approach of universal extensions (see \cite{Ri}). 
We start with $\Delta(\lambda)$ and go down with respect to 
the preorder $\preceq$. If all first extensions
from all (shifted) standard modules to $\Delta(\lambda)$ vanish,
we get $\Delta(\lambda)\in \mathcal{F}^{\downarrow}(\overline{\nabla})$
by Theorem~\ref{thm2}\eqref{thm2-2}. Otherwise there exist $\mu\in\Lambda$ and
$j'\in\mathbb{Z}$ such that 
\begin{displaymath}
\mathrm{ext}^1_A(\Delta(\mu)\langle j'\rangle,\Delta(\lambda))\neq 0. 
\end{displaymath}
We assume that $\mu$ is maximal with such property 
(we have $\mu\prec\lambda$ by Lemma~\ref{lem10}) and 
use Lemma~\ref{lem9} to choose  $j'$ such that
\begin{displaymath}
\mathrm{ext}^1_A(\Delta(\nu)\langle j\rangle,\Delta(\lambda))\neq 0 
\end{displaymath}
implies $j\leq j'$ for all $\nu\in\overline{\mu}$.

For every $\nu\in\overline{\mu}$ and $j\leq j'$ the space
$\mathrm{ext}^1_A(\Delta(\nu)\langle j\rangle,\Delta(\lambda))$ is
finite dimensional, say of dimension $l_{\nu,j}$. Consider
the universal extension 
\begin{equation}\label{eqindec}
X\hookrightarrow Y\tto Z, 
\end{equation}
where $X=\Delta(\lambda)$ and
\begin{displaymath}
Z=\bigoplus_{\nu\in\overline{\mu}}\bigoplus_{j\leq j'}
\Delta(\nu)\langle j\rangle^{l_{\nu,j}}\in \mathcal{F}^{\downarrow}(\Delta)
\end{displaymath}
(note that $\mathrm{ext}_A^1(Z,Z)=0$ by Lemma~\ref{lem10}).
We have $Y\in \mathcal{F}^{\downarrow}(\Delta)$ 
by construction. We further claim that $Y$ is indecomposable. Indeed,
Let $e\in\mathrm{end}_A(Y)$ be a nonzero idempotent
(note that $e$ is homogeneous of degree zero). As
$\nu\prec\lambda$, we have $\mathrm{hom}_A(\Delta(\lambda),
\Delta(\nu)\langle j\rangle)=0$ for any $\nu$ and $j$ as above.
Therefore $e$ maps $X$ (which is indecomposable) 
to $X$. If $e\vert_{X}=0$, then  $e$ provides 
a splitting for a nontrivial direct summand of $Z$ in \eqref{eqindec}; 
if $e\vert_{X}=\mathrm{id}_{X}$ and
$e\neq \mathrm{id}_Y$, then $\mathrm{id}_Y-e\neq 0$ 
annihilates $X$ and hence provides a 
splitting for a nontrivial direct summand of $Z$ in \eqref{eqindec}. 
This contradicts our construction of $Y$ as the universal extension. 
Therefore $e=\mathrm{id}_Y$, which proves that
the module $Y$ is indecomposable. By Lemma~\ref{lem10},
there are no extensions between the summands of $Z$. 
From $\mathrm{ext}_A^1(Z,Z)=0$ and the
universality of our extension, we get
\begin{displaymath}
\mathrm{ext}^1_A(\Delta(\nu)\langle j\rangle,Y)=0
\end{displaymath}
for all $\nu\in\overline{\mu}$ and all $j$.

Now take the indecomposable module constructed in the previous paragraph 
as $X$, take a maximal $\mu'$ such that for some $j$ we have
$\mathrm{ext}^1_A(\Delta(\mu')\langle j\rangle,X)\neq0$
and do the same thing as in the previous paragraph. 
Proceed inductively. In a finite
number of steps we end up with an indecomposable module $T(\lambda)$ such that 
$\Delta(\lambda)\hookrightarrow T(\lambda)$, the cokernel is in 
$\mathcal{F}^{\downarrow}(\Delta)$, and 
\begin{displaymath}
\mathrm{ext}^1_A(\Delta(\mu)\langle j\rangle,T(\lambda))= 0
\end{displaymath}
for all $\mu$ and $j$. By Theorem~\ref{thm2}\eqref{thm2-2}, we have 
$T(\lambda)\in\mathcal{F}^{\downarrow}(\overline{\nabla})$. 
This proves the existence part of statement \eqref{thm8-2}. 
The uniqueness part will follow from statement \eqref{thm8-3}.

Let $M\in \mathcal{F}^{\downarrow}(\Delta)\cap 
\mathcal{F}^{\downarrow}(\overline{\nabla})$ be indecomposable
and $\Delta(\lambda)\hookrightarrow M$ be such that
the cokernel $\mathrm{Coker}$ has a standard filtration. Applying
$\mathrm{hom}_A({}_-,T(\lambda))$ to the short exact
sequence
\begin{displaymath}
\Delta(\lambda)\hookrightarrow M \tto \mathrm{Coker}
\end{displaymath}
we obtain the exact sequence
\begin{displaymath}
\mathrm{hom}_A(M,T(\lambda))\to
\mathrm{hom}_A(\Delta(\lambda),T(\lambda))\to 
\mathrm{ext}_A^1(\mathrm{Coker},T(\lambda)).
\end{displaymath}
Here the right term is zero by Lemma~\ref{lemnn7} and the definition
of $T(\lambda)$. As the middle term is obviously nonzero, we obtain
that the left term is nonzero as well. This gives us a nonzero
map $\alpha$ from $M$ to $T(\lambda)$. Similarly
one constructs a nonzero map $\beta$ from $T(\lambda)$
to $M$ such that the composition $\alpha\circ\beta$ is the identity 
on $\Delta(\lambda)$. We claim the following:

\begin{lemma}\label{lemnn5}
Let $T(\lambda)$ be as above.
\begin{enumerate}[(i)]
\item \label{lemnn5.1} For any $n\in\mathbb{Z}$ there exists
a submodule $N^{(n)}$ of $T(\lambda)$ with the following properties:
\begin{enumerate}[(a)]
\item\label{lemnn5.1.0}
$N^{(n)}$ is indecomposable;
\item\label{lemnn5.1.1}
$N^{(n)}$ has finite standard filtration starting with $\Delta(\lambda)$;
\item\label{lemnn5.1.2}
$N^{(n)}_i=T(\lambda)_i$ for all $i\leq n$;
\item\label{lemnn5.1.3}
every endomorphism of $T(\lambda)$ restricts to an endomorphism 
of $N^{(n)}$.
\end{enumerate}
\item \label{lemnn5.2}
The composition $\alpha\circ\beta$ is an automorphism of 
$T(\lambda)$.
\end{enumerate}
\end{lemma}

\begin{proof}
Consider the multiset $\mathcal{M}$ of all standard subquotients 
of $T(\lambda)$. It might be infinite. However, for every $m\in\mathbb{Z}$ 
the multiset $\mathcal{M}_m$ of those subquotients $X$ of $T(\lambda)$,
for which $X_i\neq 0$ for some $i\leq m$ is finite since $T(\lambda)\in
A^{\downarrow}\text{-}\mathrm{mod}$. Construct the submultiset 
$\mathcal{N}$ of $\mathcal{M}$ in the following way: start with 
$\mathcal{M}_n\cup\{\Delta(\lambda)\}$, 
which is finite. From Lemma~\ref{lem9} it follows that 
every subquotient from  $\mathcal{M}_n$ has a nonzero first extension 
with finitely many other subquotients from  $\mathcal{M}$. Add 
to $\mathcal{N}$ all such subquotients (counted with multiplicities), 
moreover, if we add some $\Delta(\mu)\langle j\rangle$, add as well
all $\Delta(\nu)\langle i\rangle$, where $i\geq j$ and $\mu\preceq\nu$, 
occurring in $\mathcal{M}$. Obviously, the result will be a finite set. 
Repeat now the same procedure for all newly
added subquotients and continue. By Lemma~\ref{lem10}, on every next 
step we will add only $\Delta(\nu)\langle i\rangle$ such that 
$\mu\prec \nu$  (strict inequality!) for some minimal $\mu$ in the set 
indexing subquotients added on the previous
step. 

As $\Lambda$ is finite, after finitely many steps 
we will get a finite submultiset $\mathcal{N}$ of $\mathcal{M}$ 
with the following properties:
any subquotient from $\mathcal{N}$ does not extend any subquotient
from $\mathcal{M}\setminus\mathcal{N}$; there are no 
homomorphisms from any subquotient from $\mathcal{N}$ to any subquotient
from $\mathcal{M}\setminus\mathcal{N}$. Using the vanishing of the
first extension one shows that there is a submodule $N^{(n)}$  of
$T(\Lambda)$, which has a standard filtration with the multiset of subquotients
being precisely $\mathcal{N}$, in particular, $N^{(n)}$ satisfies 
\eqref{lemnn5.1.1}. By construction, $N^{(n)}$ also satisfies 
\eqref{lemnn5.1.2}. The vanishing of homomorphisms from 
subquotients from $\mathcal{N}$ to subquotients
from $\mathcal{M}\setminus\mathcal{N}$ implies that $N^{(n)}$ satisfies 
\eqref{lemnn5.1.3}. That $N^{(n)}$ satisfies \eqref{lemnn5.1.0} is proved
similarly to the proof of the indecomposability of $T(\lambda)$.
This proves statement \eqref{lemnn5.1}.

To prove that $\alpha\circ\beta$ is an automorphism 
(statement \eqref{lemnn5.2}) it is enough to show that for any
$n\in\mathbb{Z}$ the restriction of $\alpha\circ\beta$ to
$T(\lambda)_n$ is a linear automorphism.
The restriction of $\alpha\circ\beta$ to  $N^{(n)}$ (which is well defined 
by \eqref{lemnn5.1.3}) is not nilpotent as it is the
identity on $\Delta(\lambda)$. As $A$ is positively graded, the space
$\mathrm{hom}_A(\Delta(\mu),\Delta(\nu)\langle j\rangle)$ is finite
dimensional for all $\mu,\nu$ and $j$. From this observation and
\eqref{lemnn5.1.1} it follows that the endomorphism
algebra of $N^{(n)}$ is finite dimensional. This algebra is local
by \eqref{lemnn5.1.0}. Therefore the restriction of 
$\alpha\circ\beta$ to  $N^{(n)}$, being a non-nilpotent element of
a local finite dimensional algebra, is an automorphism. Therefore
the restriction of $\alpha\circ\beta$ to all $N^{(n)}_i$, in particular,
to $N^{(n)}_n=T(\lambda)_n$ (see \eqref{lemnn5.1.2}), is a linear
automorphism. This completes the proof.
\end{proof}

After Lemma~\ref{lemnn5},
substituting  $\alpha$ by $(\alpha\circ\beta)^{-1}\circ\alpha$, we may
assume that $\alpha\circ\beta=\mathrm{id}_{T(\lambda)}$. We also have that
$\beta$ is injective and $\alpha$ is surjective. The gives us splittings
for the following two short exact sequences:
\begin{gather*}
\xymatrix{
0\ar[rr]&&\mathrm{Ker}(\alpha)\ar@{^{(}->}[rr] &&M\ar[rr]_{\alpha}
&&T(\lambda)\ar[rr]\ar@/_1pc/@{-->}[ll]_{\beta}&&0
}
\\ 
\xymatrix{
0\ar[rr]&&T(\lambda)\ar[rr]_{\beta}&&M\ar@{->>}[rr]
\ar@/_1pc/@{-->}[ll]_{\alpha}&&\mathrm{Coker}(\beta)\ar[rr]&&0
}
\end{gather*}
As $M$ is assumed to be indecomposable, we obtain 
$\mathrm{Ker}(\alpha)=\mathrm{Coker}(\beta)=0$, which implies
that $\alpha$ and $\beta$ are isomorphisms. Therefore $M\cong T(\lambda)$,
which completes the proof of the theorem.
\end{proof}

The objects of the category $\mathcal{F}^{\downarrow}(\Delta)\cap 
\mathcal{F}^{\downarrow}(\overline{\nabla})$ are called
{\em tilting modules}. 

\begin{remark}\label{rem16}
{\rm  
Note that a tilting module may be an infinite
direct sum of indecomposable tilting modules. 
Note also that the direct sum of all indecomposable tilting
modules (with all shifts) does not belong 
to $A^{\downarrow}\text{-}\mathrm{gmod}$. It might happen that it
does not belong to $A\text{-}\mathrm{gmod}$ either, since local 
finiteness is an issue.
}
\end{remark}

\begin{corollary}\label{cor15}
Let $A$ be a positively graded standardly stratified algebra.
\begin{enumerate}[(i)]
\item\label{cor15-1}
Every $M\in \mathcal{F}^{\downarrow}(\Delta)$ has a coresolution by 
tilting modules of length at most $|\Lambda|-1$. 
\item\label{cor15-2}
Every $M\in \mathcal{F}^{\downarrow}(\overline{\nabla})$ has a 
(possibly infinite) resolution by tilting modules.
\end{enumerate}
\end{corollary}

\begin{proof}
This follows from Theorem~\ref{thm8} and
the definitions by standard arguments.
\end{proof}

\begin{remark}\label{rem14}
{\rm  
Note that the standard filtration of $T(\lambda)$ may be infinite,
see Example~\ref{exm2}.
}
\end{remark}

Unfortunately, Remark~\ref{rem14} says that one can not hope for
a reasonable analogue of Ringel duality on the class of algebras we
consider. We can of course consider the  endomorphism 
algebra of the direct  sum of all tilting modules, but from 
Remark~\ref{rem14} it follows that projective modules over such 
algebras might have infinite standard filtrations and hence we will
not be able to construct tilting modules for them.
Another obstruction is that we actually can not guarantee that the 
induced grading on this endomorphism algebra will be positive
(see examples in \cite{MO,Ma5}). To deal with these problems we have
to introduce some additional restrictions.

\section{Ringel duality for graded standardly stratified 
algebras}\label{s31}

Consider the $\Bbbk$-linear category $\mathfrak{T}$, which is the
full subcategory of $A^{\downarrow}\text{-}\mathrm{gmod}$, whose objects
are $T(\lambda)\langle j\rangle$, where $\lambda\in\Lambda$ and 
$j\in\mathbb{Z}$. The group $\mathbb{Z}$ acts freely on $\mathfrak{T}$ via
$\langle j\rangle$ and the quotient of $\mathfrak{T}$ modulo
this free action is a $\mathbb{Z}$-graded $\Bbbk$-linear 
category $\overline{\mathfrak{T}}$, whose objects can be identified
with $T(\lambda)$, where $\lambda\in\Lambda$ 
(see \cite{DM,MOS} for more details).
Thus the ungraded endomorphism algebra $R(A)=\mathrm{End}_A(T)$, 
where  $T=\bigoplus_{\lambda\in\Lambda}T(\lambda)$ becomes a
$\mathbb{Z}$-graded  $\Bbbk$-algebra in the
natural way. The algebra $R(A)$ is called the {\em Ringel dual} of $A$.
The algebra $A$ will be called {\em weakly adapted} provided 
that every $T(\lambda)$, where $\lambda\in\Lambda$, has a finite standard 
filtration.  The algebra $A$ will be called {\em adapted} provided that
the above $\mathbb{Z}$-grading on $R(A)$ is positive.

\begin{proposition}\label{prop31}
We have the following:
\begin{enumerate}[(i)]
\item\label{prop31.1} Any adapted algebra is weakly adapted.
\item\label{prop31.2} If $A$ is weakly adapted, then $R(A)$ is locally 
finite.
\end{enumerate}
\end{proposition}

\begin{proof}
Because of Lemma~\ref{lem5} and the definition of tilting modules,
every homomorphism from $T(\lambda)$ to $T(\mu)\langle j\rangle$ is
induced from a homomorphism from some standard subquotient of
$T(\lambda)$ to some proper standard subquotient of 
$T(\mu)\langle j\rangle$.

Since $\overline{\nabla}(\mu)\langle j\rangle$ is a (sub)quotient of 
$T(\mu)\langle j\rangle$, the condition that 
the above $\mathbb{Z}$-grading on $R(A)$ is positive implies that 
every standard subquotient of $T(\lambda)$, different from 
$\Delta(\lambda)$, must have the form $\Delta(\mu)\langle j\rangle$
for some $j>0$. However, the vector space 
$\displaystyle \bigoplus_{j\leq 0}T(\lambda)_j$ is finite dimensional
as $T(\lambda)\in A^{\downarrow}\text{-}\mathrm{gmod}$, which
yields that any standard filtration of $T(\lambda)$ must be finite.
This proves statement \eqref{prop31.1}. 

Statement \eqref{prop31.2} follows from the finiteness of a standard
filtration of $T(\lambda)$ and the obvious fact that 
$\mathrm{hom}_A(\Delta(\lambda),M)$ is finite dimensional
for any $M\in A\text{-}\mathrm{gmod}$.
\end{proof}

\begin{corollary}\label{cor301}
Assume that $A$ is adapted. Then every $M\in\mathcal{F}^b(\Delta)$,
in particular, every indecomposable projective
$A$-module, has a finite coresolution
\begin{equation}\label{eq333}
0\to M\to T_0\to T_1\to\dots \to T_k\to 0, 
\end{equation}
such that every $T_i$ is a finite direct sum of indecomposable tilting
$A$-modules.
\end{corollary}

\begin{proof}
It is enough to prove the claim for $M=\Delta(\lambda)$. 
The claim is obvious in the case $\lambda$ is minimal as in this 
case we have $\Delta(\lambda)=T(\lambda)$. From
Theorem~\ref{thm8}\eqref{thm8-2} we have the
exact sequence
\begin{displaymath}
0\to  \Delta(\lambda)\to T(\lambda)\to \mathrm{Coker}
\end{displaymath}
such that $\mathrm{Coker}$ has a standard filtration with possible
subquotients $\Delta(\mu)\langle i\rangle$, where $\mu\prec\lambda$ and 
$i\in\mathbb{Z}$. By Proposition~\ref{prop31}\eqref{prop31.1}, 
the standard filtration of
$\mathrm{Coker}$ is finite and hence the claim follows by induction
(with respect to the partial preorder $\preceq$).
\end{proof}

A complex $\mathcal{X}^{\bullet}$ of $A$-modules is called {\em perfect} 
provided that it is bounded and every nonzero $\mathcal{X}^{i}$ is a direct 
sum of finitely many indecomposable modules.
Let $\mathcal{P}(A)$ denote the homotopy category of perfect complexes
of graded projective $A$-modules. As every 
indecomposable projective $A$-module has a finite standard filtration,
it follows by induction that $\mathcal{F}^b(\Delta)\subseteq \mathcal{P}(A)$. 
Consider the contravariant functor 
\begin{displaymath}
\mathrm{G}=\mathcal{R}\mathrm{hom}_A({}_-,\mathfrak{T})
\end{displaymath}
(see \cite{MOS} for details of hom-functors for $\Bbbk$-linear
categories). As we will see in Theorem~\ref{thm32}\eqref{thm32-3}, the
functor $\mathrm{G}$ is a functor from $\mathcal{P}(A)$ to 
$\mathcal{P}(R(A))$. To distinguish 
$A$ and $R(A)$-modules, if necessary, we will use $A$ and $R(A)$
as superscripts for the corresponding modules.

\begin{theorem}[Weak Ringel duality]
\label{thm32}
Let $A$ be an adapted standardly stratified algebra.
\begin{enumerate}[(i)]
\item\label{thm32-1}
The algebra $R(A)$ is an adapted standardly stratified algebra with
respect to $\preceq^{\mathrm{op}}$. 
\item\label{thm32-2}
We have $R(R(A))\cong A$.
\item\label{thm32-3}
The functor $\mathrm{G}$ is an antiequivalence from 
$\mathcal{P}(A)$ to $\mathcal{P}(R(A))$.
\item\label{thm32-4} 
The functor $\mathrm{G}$ induces an antiequivalence between 
$\mathcal{F}^{b}(\Delta^{(A)})$ and 
$\mathcal{F}^{b}(\Delta^{(R(A))})$, which sends standard
$A$-modules to standard $R(A)$-modules, tilting $A$-modules to 
projective $R(A)$-modules and projective $A$-modules
modules to tilting $R(A)$-modules.
\end{enumerate}
\end{theorem}

\begin{proof}
By construction, the functor $\mathrm{G}$ maps indecomposable tilting
$A$-modules to indecomposable projective $R(A)$-modules. From
Corollary~\ref{cor301} it follows that every indecomposable projective
$A$-module $M$ has a coresolution of the form \eqref{eq333},
such that every $T_i$ is a finite direct sum of indecomposable tilting
$A$-modules. This implies that every object in $\mathcal{P}(A)$ can
be represented by a perfect complex of tilting modules. 
This yields that $\mathrm{G}$ maps 
$\mathcal{P}(A)$ to $\mathcal{P}(R(A))$. As $T$ is a tilting module,
statement  \eqref{thm32-3} follows directly from the Rickard-Morita 
Theorem for $\Bbbk$-linear categories, see e.g. \cite[Corollary~9.2]{Ke}
or \cite[Theorem~2.1]{DM}.

The functor $\mathrm{G}$ is acyclic, in particular, exact on 
$\mathcal{F}^{b}(\Delta^{(A)})$ by  Lemma~\ref{lem5}. By construction, 
it maps tilting $A$-modules to projective $R(A)$-modules and thus projective
$R(A)$-modules have filtrations by images (under $\mathrm{G}$) of 
standard $A$-modules. By Proposition~\ref{prop31},
these filtrations of projective $R(A)$-modules by images 
of  standard $A$-modules are finite. As in the classical case (see \cite{Ri})
it is easy to see that the images of standard $A$-modules are
standard $R(A)$-modules (with respect to $\preceq^{\mathrm{op}}$). 
From Proposition~\ref{prop31}\eqref{prop31.2} and our assumptions
it follows that the algebra $R(A)$ is
positively graded. This implies that 
$R(A)$ is a graded standardly stratified algebra (with respect 
to $\preceq^{\mathrm{op}}$). 

Because of our description of standard modules for $R(A)$, the 
functor $\mathrm{G}$ maps $\mathcal{F}^{b}(\Delta^{(A)})$ to
$\mathcal{F}^{b}(\Delta^{(R(A))})$. In particular, projective
$A$-modules are also mapped to some modules in 
$\mathcal{F}^{b}(\Delta^{(R(A))})$. Since $\mathrm{G}$ is
a derived equivalence by \eqref{thm32-3}, for $i>0$, $j\in\mathbb{Z}$
and $\lambda,\mu\in \Lambda$ we obtain 
\begin{displaymath}
\mathrm{ext}^i_{R(A)}(\mathrm{G}\Delta(\lambda)\langle j\rangle,
\mathrm{G}P(\mu))=
\mathrm{ext}^i_{A}(P(\mu),\Delta(\lambda)\langle j\rangle)=0.
\end{displaymath}
Hence $\mathrm{G}P(\mu)$ has a proper costandard filtration by 
Theorem~\ref{thm2}\eqref{thm2-1}, and thus is a tilting 
$R(A)$-module, which implies
\eqref{thm32-2}. As projective $A$-modules have finite standard
filtration, the algebra $R(A)$ is weakly adapted. It is even
adapted as the grading on $R(R(A))$ coincides with the grading on
$A$ and is hence positive. This proves \eqref{thm32-1}. 
Statement \eqref{thm32-4} follows easily from the properties
of $\mathrm{G}$, established above. This completes the proof.
\end{proof}

Similarly to the above we consider the contravariant functors 
\begin{gather*}
\mathrm{F}=\mathcal{R}\mathrm{hom}_A(\mathfrak{T},{}_-)^{\circledast}:
\mathcal{D}^{+}(A^{\uparrow}\text{-}\mathrm{gmod})\to
\mathcal{D}^{-}(R(A)^{\downarrow}\text{-}\mathrm{gmod})\\
\tilde{\mathrm{F}}=\mathcal{R}\mathrm{hom}_A(\mathfrak{T},{}_-)^{\circledast}:
\mathcal{D}^{-}(A^{\downarrow}\text{-}\mathrm{gmod})\to
\mathcal{D}^{+}(R(A)^{\uparrow}\text{-}\mathrm{gmod}).
\end{gather*}
Although it is not obvious from the first impression, the following 
statement carries a strong resemblance with \cite[Proposition~20]{MOS}:

\begin{theorem}[Strong Ringel duality]
\label{thm33}
Let $A$ be an adapted standardly stratified algebra.
\begin{enumerate}[(i)]
\item\label{thm33-3}
Both $\mathrm{F}$ and $\tilde{\mathrm{F}}$ are antiequivalences.
\item\label{thm33-4} 
The functor $\mathrm{F}$ induces an antiequivalence from the
category $\mathcal{F}^{\uparrow}(\overline{\nabla}^{(A)})$ to the category
$\mathcal{F}^{\downarrow}(\overline{\nabla}^{(R(A))})$, which sends 
proper costandard $A$-modules to proper costandard $R(A)$-modules, 
and injective $A$-modules to tilting $R(A)$-modules.
\item\label{thm33-5} 
The functor $\tilde{\mathrm{F}}$ induces an antiequivalence from the
category $\mathcal{F}^{\downarrow}(\overline{\nabla}^{(A)})$ to the category
$\mathcal{F}^{\uparrow}(\overline{\nabla}^{(R(A))})$, which sends 
proper costandard $A$-modules to proper costandard $R(A)$-modules, 
and tilting $A$-modules to injective $R(A)$-modules.
\end{enumerate}
\end{theorem}

\begin{proof}
Consider the covariant versions of our functors:
\begin{gather*}
\mathrm{H}=\mathcal{R}\mathrm{hom}_A(\mathfrak{T},{}_-):
\mathcal{D}^{+}(A^{\uparrow}\text{-}\mathrm{gmod})\to
\mathcal{D}^{+}(\mathrm{gmod}\text{-}R(A)^{\uparrow})\\
\tilde{\mathrm{H}}=\mathcal{R}\mathrm{hom}_A(\mathfrak{T},{}_-):
\mathcal{D}^{-}(A^{\downarrow}\text{-}\mathrm{gmod})\to
\mathcal{D}^{-}(\mathrm{gmod}\text{-}R(A)^{\downarrow}).
\end{gather*}
Every object in $\mathcal{D}^{-}(A^{\downarrow}\text{-}\mathrm{gmod})$ 
has a projective resolution. Since $T$ is a tilting module, every object in 
$\mathcal{D}^{-}(A^{\downarrow}\text{-}\mathrm{gmod})$ is also given
by a complex of tilting modules. As tilting modules are selforthogonal,
for complexes of tilting modules the functor $\tilde{\mathrm{H}}$ reduces
to the usual hom functor. Similarly every object in 
$\mathcal{D}^{+}(A^{\uparrow}\text{-}\mathrm{gmod})$ has an
injective resolution and for such complexes the functor $\mathrm{H}$ 
reduces to the usual hom functor.

The left adjoints $\mathrm{H}'$ and $\tilde{\mathrm{H}}'$
of $\mathrm{H}$ and $\tilde{\mathrm{H}}$, respectively, are thus given
by the left derived of the tensoring with $\mathfrak{T}$. As $T$ is
a tilting module, these left adjoin´t functors can be given as
a tensoring with a finite tilting complex of $A\text{-}R(A)$-bimodules,
projective as right $R(A)$-modules, followed by taking the total complex.

Using the definition of proper costandard modules it is 
straightforward to verify that both $\mathrm{H}$ and $\tilde{\mathrm{H}}$
map  proper costandard left $A$-modules to proper standard right $R(A)$-modules.
Similarly, both  $\mathrm{H}'$ and $\tilde{\mathrm{H}}'$
map proper standard right $R(A)$-modules to proper costandard left 
$A$-modules. Since proper (co)standard objects have trivial endomorphism 
rings, it follows by standard arguments that the adjunction morphism
\begin{gather*}
\mathrm{Id}_{\mathcal{D}^{+}(\mathrm{gmod}\text{-}R(A)^{\uparrow})}
\to \mathrm{H}\mathrm{H}',\quad
\mathrm{H}'\mathrm{H}\to
\mathrm{Id}_{\mathcal{D}^{+}(A^{\uparrow}\text{-}\mathrm{gmod})} \\
\mathrm{Id}_{\mathcal{D}^{-}(\mathrm{gmod}\text{-}R(A)^{\downarrow})}
\to \tilde{\mathrm{H}}\tilde{\mathrm{H}}',\quad
\tilde{\mathrm{H}}'\tilde{\mathrm{H}}\to
\mathrm{Id}_{\mathcal{D}^{-}(A^{\downarrow}\text{-}\mathrm{gmod})} \\
\end{gather*}
induce isomorphisms, when evaluated on respective proper (co)standard 
objects. Therefore the adjunction morphism above are isomorphisms 
of functors on the categories, generated (as triangular categories)
by proper (co)standard objects. Using the classical limit construction
(see \cite{Ric}) one shows that both $\mathrm{H}$ and $\tilde{\mathrm{H}}$ 
are equivalences of categories. 
This yields that both $\mathrm{F}$ and $\tilde{\mathrm{F}}$ are  
antiequivalences of categories. This proves statement \eqref{thm33-3} and 
statements \eqref{thm33-4} and \eqref{thm33-5} easily follow.
\end{proof}

\section{Proof of the main result}\label{s4}

If $M\in\{P(\lambda),I(\lambda),T(\lambda),\Delta(\lambda),
\overline{\nabla}(\lambda)\}$, we will say that the centroid of the 
graded modules $M\langle j\rangle$, where $j\in\mathbb{Z}$, belongs to $-j$.
Let $\mathcal{X}^{\bullet}$ and $\mathcal{Y}^{\bullet}$ be two
complexes of tilting modules, both bounded from the right. 
A complex $\mathcal{X}^{\bullet}$ of projective, injective, tilting,
standard, or costandard modules is called {\em linear} provided that for
every $i$ centroids of all indecomposable summand of 
$\mathcal{X}^{i}$ belong to $-i$. A positively graded algebra $B$ is 
called {\em Koszul} if all simple $B$-modules have linear projective 
resolutions. The Koszul dual $E(A)$ of a Koszul algebra $A$ is just the 
Yoneda extension algebra of the direct sum of all simple $A$-modules. 
The algebra $E(A)$ is positively graded by the degree of extensions.

We will say that $\mathcal{X}^{\bullet}$ {\em dominates} 
$\mathcal{Y}^{\bullet}$ provided that for every $i\in\mathbb{Z}$ the 
following holds: if the centroid of an indecomposable summand of 
$\mathcal{X}^{i}$ belongs to $j$ and the centroid of an indecomposable 
summand of $\mathcal{Y}^{i}$  belongs to $j'$, then $j<j'$.

The aim of this section is to prove Theorem~\ref{thm1}. For this
we fix an algebra $A$ satisfying the assumptions of Theorem~\ref{thm1}
throughout (we will call such algebra {\em balanced}). 
For $\lambda\in\Lambda$ we denote by
$\mathcal{S}_{\lambda}^{\bullet}$ and  $\mathcal{C}_{\lambda}^{\bullet}$ 
the linear tilting coresolution of $\Delta(\lambda)$ and resolution 
of $\overline{\nabla}(\lambda)$, respectively. We will proceed along the 
lines of \cite[Section~3]{Ma5} and do not repeat the arguments, which are
similar to the ones from \cite[Section~3]{Ma5}.

\begin{lemma}\label{nlem2}
The algebra $A$ is adapted.
\end{lemma}

\begin{proof}
Mutatis mutandis \cite[Lemma~2]{Ma5}.
\end{proof}

\begin{corollary}\label{ncor3}
We have $\mathrm{hom}_A(T(\lambda)\langle i\rangle,T(\mu))=0$,
for all $\lambda,\mu\in\Lambda$ and $i\in\mathbb{N}$.
\end{corollary}

\begin{corollary}\label{ncor4}
Let $\mathcal{X}^{\bullet}$ and $\mathcal{Y}^{\bullet}$ be two 
complexes of tilting modules, both bounded from the right. Assume that 
$\mathcal{X}^{\bullet}$ dominates $\mathcal{Y}^{\bullet}$.
Then  $\mathrm{Hom}_{\mathcal{D}^{-}(A)}(\mathcal{X}^{\bullet},
\mathcal{Y}^{\bullet})=0$.
\end{corollary}

\begin{proof}
Mutatis mutandis \cite[Corollary~4]{Ma5}.
\end{proof}

\begin{proposition}\label{nprop5}
For every $\lambda\in\Lambda$ the module $L(\lambda)$ is isomorphic
in $\mathcal{D}^{-}(A)$ to a linear complex 
$\mathcal{L}_{\lambda}^{\bullet}$ of tilting modules.
\end{proposition}

\begin{proof}
Just as in \cite[Proposition~5]{Ma5}, one constructs a complex $\overline{\mathcal{P}}^{\bullet}$ of tilting modules in
$\mathcal{D}^{-}(A)$, quasi-isomorphic to $L(\lambda)$ and such  that
for each $i$ all centroids  of indecomposable summands in 
$\overline{\mathcal{P}}^{i}$ belong to some $j$ such that $j\geq -i$.

Let us now prove the claim by induction with respect to $\preceq$.
If $\lambda$ is minimal, then $L(\lambda)=\overline{\nabla}(\lambda)$
and we can take $\mathcal{L}_{\lambda}^{\bullet}=
\mathcal{C}_{\lambda}^{\bullet}$. Otherwise, consider the
short exact sequence
\begin{displaymath}
0\to L(\lambda)\to  \overline{\nabla}(\lambda)\to
\mathrm{Coker}\to 0.
\end{displaymath}
Since $A$ is positively graded, we have $\mathrm{Coker}_j=0$ for
all $j\geq 0$. Moreover, $\mathrm{Coker}$ is finite dimensional 
(Lemma~\ref{lem4}) and
all simple subquotients of $\mathrm{Coker}$ correspond to 
some $\mu\in\Lambda$ such that $\mu\prec \lambda$. Using the
inductive assumption, we can resolve every simple subquotient of
$\mathrm{Coker}$ using the corresponding linear complexes of tilting modules
and thus obtain that $\mathrm{Coker}$ is quasi-isomorphic to some
complex $\mathcal{X}^{\bullet}$ of tilting modules such that for each 
$i$ all centroids  of indecomposable summands in  $\mathcal{X}^{i}$ 
belong to some $j$ such that $j\leq -i-1$. As $\overline{\nabla}(\lambda)$
has a linear tilting resolution, it follows that $L(\lambda)$ is 
quasi-isomorphic to some complex $\overline{Q}^{\bullet}$ of 
tilting modules,  such  that for each $i$ all centroids  of 
indecomposable summands in 
$\overline{\mathcal{Q}}^{i}$ belong to some $j$ such that $j\leq -i$.

Because of the uniqueness of the minimal tilting complex
$\mathcal{L}_{\lambda}^{\bullet}$, representing $L(\lambda)$ in 
$\mathcal{D}^{-}(A^{\downarrow}\text{-}\mathrm{mod})$, 
we thus conclude that for all $i\in\mathbb{Z}$
centroids of all indecomposable summands in 
$\mathcal{L}_{\lambda}^{i}$ belong to $-i$. This means that
$\mathcal{L}_{\lambda}^{\bullet}$ is linear and completes the proof.
\end{proof}

\begin{corollary}\label{ncor6}
The algebra $A$ is Koszul.
\end{corollary}

\begin{proof}
Mutatis mutandis \cite[Corollary~6]{Ma5}.
\end{proof}

\begin{corollary}\label{ncor7}
We have the following:
\begin{enumerate}[(i)]
\item\label{ncor7-1} Standard $A$-modules have linear projective
resolutions.
\item\label{ncor7-2} Proper costandard $A$-modules have linear 
injective coresolutions.
\end{enumerate}
\end{corollary}

\begin{proof}
Assume that $\mathrm{ext}^i_A(\Delta(\lambda),L(\mu)\langle j\rangle)\neq 0$
for some $\lambda,\mu\in\Lambda$, $i\geq 0$ and $j\in\mathbb{Z}$. 
As $A$ is positively graded we obviously have $j\leq -i$.
On the other hand, this inequality yields an existence of a 
non-zero homomorphism (in 
$\mathcal{D}^{-}(A^{\downarrow}\text{-}\mathrm{mod})$) from 
$\mathcal{S}_{\lambda}^{\bullet}$ to 
$\mathcal{L}_{\lambda}^{\bullet}[i]\langle j\rangle$. But both
$\mathcal{S}_{\lambda}^{\bullet}$ and $\mathcal{L}_{\lambda}^{\bullet}$
are linear (Proposition~\ref{nprop5})
and hence from Corollary~\ref{ncor4} it follows that $j\geq -i$.
Therefore $j=-i$ and statement \eqref{ncor7-1} follows. The statement
\eqref{ncor7-2} is proved similarly.
\end{proof}

\begin{corollary}\label{ncor8}
We have the following:
\begin{enumerate}[(i)]
\item \label{ncor8-1} Standard $R(A)$-modules have finite linear projective
resolutions. 
\item \label{ncor8-2} Standard $R(A)$-modules have finite linear tilting
coresolutions. 
\item \label{ncor8-3} Proper costandard $R(A)$-modules have linear 
tilting resolutions. 
\item \label{ncor8-4} Proper costandard $R(A)$-modules have linear 
injective coresolutions. 
\end{enumerate}
\end{corollary}

\begin{proof}
Using Theorem~\ref{thm32}\eqref{thm32-4} we see that the functor  
$\mathrm{G}$ maps a finite linear projective resolution of  $\Delta^{(A)}$ 
(Corollary~\ref{ncor7}\eqref{ncor7-1}) to a finite linear tilting coresolution 
of $\Delta^{(R(A))}$. It also maps a finite linear tilting coresolution of 
$\Delta^{(A)}$ to a finite linear projective resolution of 
$\Delta^{(R(A))}$.

Using Theorem~\ref{thm33}\eqref{thm33-4} we see that the functor  
$\mathrm{F}$ maps a 
linear injective coresolution of  $\overline{\nabla}^{(A)}$ 
(Corollary~\ref{ncor7}\eqref{ncor7-2}) to a linear tilting resolution 
of $\overline{\nabla}^{(R(A))}$. 
Using Theorem~\ref{thm33}\eqref{thm33-5} we see that the functor  
$\tilde{\mathrm{F}}$ maps  a linear tilting resolution of 
$\overline{\nabla}^{(A)}$ to a linear injective coresolution of 
$\overline{\nabla}^{(R(A))}$. The claim follows.
\end{proof}

\begin{corollary}\label{ncor805}
The algebra $R(A)$ is Koszul.
\end{corollary}

\begin{proof}
This follows from Corollaries~\ref{ncor6} and Corollaries~\ref{ncor8}.
\end{proof}

Denote by $\mathfrak{LT}$ the full subcategory of 
$\mathcal{D}^{-}(A)$, which consists of all linear complexes of
tilting $A$-modules. The category $\mathfrak{LT}$ is equivalent to 
$\mathrm{gmod}\text{-}E(R(A))^{\uparrow}$ and simple objects of 
$\mathfrak{LT}$ have the form $T(\lambda)\langle -i\rangle[i]$, where
$\lambda\in\Lambda$ and $i\in\mathbb{Z}$ (see \cite{MOS}).

\begin{proposition}\label{nprop9}
We have the following:
\begin{enumerate}[(i)]
\item\label{nprop9-1} 
The objects $\mathcal{S}_\lambda^{\bullet}$, where $\lambda\in\Lambda$,
are proper standard objects in $\mathfrak{LT}$ with respect to $\preceq$.
\item\label{nprop9-2} 
The objects $\mathcal{C}_\lambda^{\bullet}$, where $\lambda\in\Lambda$,
are costandard objects in $\mathfrak{LT}$ with respect to $\preceq$.
\end{enumerate}
\end{proposition}

\begin{proof}
Mutatis mutandis \cite[Proposition~11]{Ma5}.
\end{proof}

\begin{proposition}\label{nprop10}
For all $\lambda,\mu\in\Lambda$ and $i,j\in\mathbb{Z}$ we have
\begin{equation}\label{eq20202}
\mathrm{Hom}_{\mathcal{D}^b(\mathfrak{LT})}
(\mathcal{S}_\lambda^{\bullet},\mathcal{C}_\mu
\langle j\rangle[-i]^{\bullet})=
\begin{cases}
\Bbbk,& \lambda=\mu, i=j=0;\\
0,& \text{otherwise}.
\end{cases}
\end{equation}
\end{proposition}

\begin{proof}
Mutatis mutandis \cite[Proposition~12]{Ma5}.
\end{proof}

\begin{corollary}\label{ncor11}
The algebra $E(R(A))$ is standardly stratified with respect to $\preceq$.
\end{corollary}

\begin{proof}
Applying the duality to Propositions~\ref{nprop9} and \ref{nprop10} we
obtain  that standard $E(R(A))$-modules are left orthogonal to proper 
costandard. Using this and the same arguments as in the proof of 
Theorem~\ref{thm2} one shows that projective $E(R(A))$-modules have a 
standard filtration.

Since standard $E(R(A))$-modules are left orthogonal to proper costandard
modules, to prove that the standard filtration of an indecomposable projective
$E(R(A))$-module is finite it is enough to show that the dimension of
the full ungraded homomorphism space from any indecomposable projective
$E(R(A))$-module to any proper costandard module is finite. In terms of
the category $\mathfrak{LT}$ (which gives the dual picture), we thus
have to show that the dimension $N$ of the full ungraded homomorphism space 
from $\mathcal{S}_\lambda^{\bullet}$ to any injective object in 
$\mathfrak{LT}$ is finite. Realizing $\mathfrak{LT}$ as linear complexes
of projective $R(A)$-modules, we know that injective objects of 
$\mathfrak{LT}$ are linear projective resolutions of simple
$R(A)$-modules (see \cite[Proposition~11]{MOS}), while the proper standard 
objects are linear projective resolutions of standard $R(A)$-modules.
We  thus get that $N$ is bounded by the sum of the dimensions of all 
extension from the corresponding standard module to the corresponding
simple module. Now the claim follows from the fact that all standard 
$R(A)$-modules have finite linear resolutions 
(Corollary~\ref{ncor8}\eqref{ncor8-1}).
\end{proof}

\begin{corollary}\label{ncor12}
The complexes $\mathcal{L}_{\lambda}^{\bullet}$, where $\lambda\in\Lambda$, 
are tilting objects in $\mathfrak{LT}$.
\end{corollary}

\begin{proof}
Mutatis mutandis \cite[Corollary~14]{Ma5}.
\end{proof}

\begin{corollary}\label{ncor14}
There is an isomorphism $E(A)\cong R(E(R(A)))$ of graded algebras, both considered with respect to the natural grading induced from 
$\mathcal{D}^-(A)$. In particular, we have $R(E(A))\cong E(R(A))$.
\end{corollary}

\begin{proof}
Mutatis mutandis \cite[Corollary~15]{Ma5}.
\end{proof}

\begin{corollary}\label{ncor15}
Both $E(A)$ and $R(E(A))$ are positively graded with respect to
the natural grading induced from  $\mathcal{D}^-(A)$. 
\end{corollary}

\begin{proof}
Mutatis mutandis \cite[Corollary~16]{Ma5}.
\end{proof}

\begin{lemma}\label{nlem17}
The algebra $E(R(A))$ is standard Koszul.
\end{lemma}

\begin{proof}
Mutatis mutandis \cite[Lemma~18]{Ma5}.
\end{proof}

\begin{proposition}\label{nprop16}
The positively graded algebras $E(A)$ and $R(E(A))$ are balanced.
\end{proposition}

\begin{proof}
Mutatis mutandis \cite[Proposition~17]{Ma5}.
\end{proof}

\begin{proof}[Proof of Theorem~\ref{thm1}.]
Statement \eqref{thm1-1} follows from Corollaries~\ref{ncor6} and \ref{ncor7}.
Statement \eqref{thm1-2} follows from Corollary~\ref{ncor8} and 
Proposition~\ref{nprop16}. 
Statement \eqref{thm1-3} follows from Proposition~\ref{nprop5}. 
Finally, statement \eqref{thm1-4} follows from Corollary~\ref{ncor14}.
\end{proof}

\section{Examples}\label{s5}

\begin{example}\label{exm1}
{\rm  
Consider the path algebra $A$ of the following quiver:
\begin{displaymath}
\xymatrix{
1\ar@(ul,dl)[]_{\alpha}\ar[rrr]^{\beta}&&&2
}
\end{displaymath}
It is positively graded in the natural way (each arrow has degree one).
We have $\Delta(2)=P(2)=L(2)$, while the projective module
$P(1)$ looks as follows:
\begin{displaymath}
\xymatrix{ 
1\ar[rrd]^{\beta}\ar[d]^{\alpha}&&\\
1\ar[rrd]^{\beta}\ar[d]^{\alpha}&&2\\
1\ar[rrd]^{\beta}\ar[d]^{\alpha}&&2\\
\vdots&&\vdots
}
\end{displaymath}
In particular, we have that the ungraded composition multiplicity of 
$L(2)$  in $P(1)$ is infinite and hence $P(1)$ has an infinite 
standard filtration.  In particular, Lemma~\ref{lem9} fails in this case
and hence the universal extension procedure does not have a starting
point and can not give us a module from 
$A^{\downarrow}\text{-}\mathrm{gmod}$.
}
\end{example}

\begin{example}\label{exm2}
{\rm  
Consider the path algebra $B$ of the following quiver:
\begin{displaymath}
\xymatrix{
1\ar[rrr]^{\alpha}&&&2\ar@(ur,dr)[]^{\beta}
}
\end{displaymath}
It is positively graded in the natural way (each arrow has degree one).
We have $\Delta(1)=L(1)=T(1)$, $\Delta(2)=P(2)$ and the following
projective $B$-modules:
\begin{displaymath}
P(1):\quad\xymatrix{ 
1\ar[d]^{\alpha}\\
2\ar[d]^{\beta}\\
2\ar[d]^{\beta}\\
\vdots
}\quad\quad\quad\quad\quad P(2):\quad
\xymatrix{ 
2\ar[d]^{\beta}\\
2\ar[d]^{\beta}\\
2\ar[d]^{\beta}\\
\vdots
}
\end{displaymath}
The module $T(2)$ looks as follows:
\begin{displaymath}
T(2):\quad\xymatrix{ 
1\ar[rrd]^{\alpha}&& \\
1\ar[rrd]^{\alpha}&& 2\ar[d]^{\beta}\\
1\ar[rrd]^{\alpha}&& 2\ar[d]^{\beta}\\
\vdots && \vdots
}
\end{displaymath}
In particular, $T(2)$ has an infinite standard filtration
and hence the algebra $B$ is not weakly adapted. 
}
\end{example}

\begin{example}\label{exm3}
{\rm  
Consider the path algebra $C$ of the following quiver:
\begin{displaymath}
\xymatrix{
1\ar[rrr]^{\alpha}\ar@(ul,dl)[]_{\beta}&&&2\ar@(ur,dr)[]^{\beta}
}
\end{displaymath}
modulo the ideal, generated by the relation $\alpha\beta=\beta\alpha$.
It is positively graded in the natural way (each arrow has degree one).
We have $\overline{\nabla}(1)=L(1)$  and also the following 
projective, standard, proper costandard and tilting  $C$-modules:
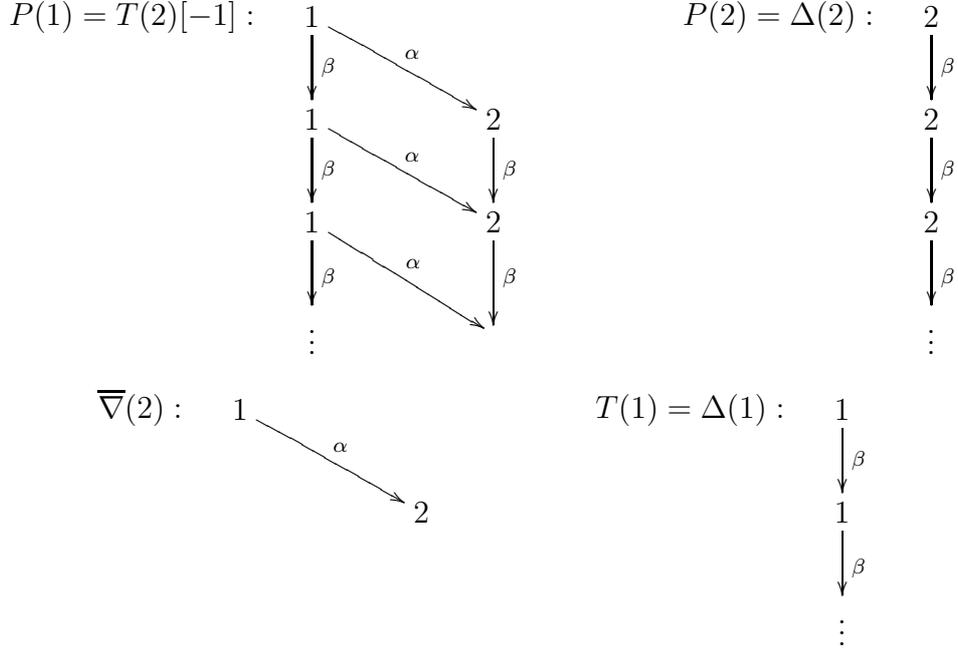
\begin{displaymath}
P(1)=T(2)[-1]:\quad\xymatrix{ 
1\ar[d]^{\beta}\ar[drr]^{\alpha}&&\\
1\ar[d]^{\beta}\ar[drr]^{\alpha}&&2\ar[d]^{\beta}\\
1\ar[d]^{\beta}\ar[drr]^{\alpha}&&2\ar[d]^{\beta}\\
\vdots&&
}\quad\quad\quad\quad\quad P(2)=\Delta(2):\quad
\xymatrix{ 
2\ar[d]^{\beta}\\
2\ar[d]^{\beta}\\
2\ar[d]^{\beta}\\
\vdots
}
\end{displaymath}
\begin{displaymath}
\overline{\nabla}(2):\quad\xymatrix{ 
1\ar[rrd]^{\alpha}&&\\&&2
}\quad\quad\quad\quad\quad
T(1)=\Delta(1):\quad\xymatrix{ 
1\ar[d]^{\beta}\\
1\ar[d]^{\beta}\\
\vdots 
}
\end{displaymath}
Standard and proper costandard $C$-modules have the following 
linear tilting (co)resolutions:
\begin{gather*}
0\to \Delta(1)\to T(1)\to 0\\
0\to \Delta(2)\to T(2)\to T(1)[1]\to 0\\
0\to T(1)[-1]\to T(1)\to \overline{\nabla}(1)\to 0\\
0\to T(2)[-1]\to T(2)\to \overline{\nabla}(2)\to 0.
\end{gather*}
Hence $C$ is balanced. The indecomposable tilting objects in 
$\mathfrak{LT}$ are:
\begin{gather*}
0\to T(1)[-1]\to T(1)\to  0\\
0\to T(2)[-1]\to T(2)\oplus T(1)\to T(1)[1]\to 0.
\end{gather*}
We have $R(C)\cong C^{\mathrm{op}}$,
$E(C)$ is the path algebra of the quiver:
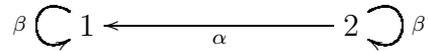
\begin{displaymath}
\xymatrix{
1\ar@(ul,dl)[]_{\beta}&&&2\ar@(ur,dr)[]^{\beta}\ar[lll]^{\alpha}
}
\end{displaymath}
modulo the ideal, generated by the relation $\alpha\beta=\beta\alpha$
and $\beta^2=0$, and $R(E(C))\cong E(R(C))\cong E(C)^{\mathrm{op}}$.
}
\end{example}

\begin{example}\label{exm4}
{\rm  
Every Koszul positively graded local algebra algebra $A$ with
$\dim_{\Bbbk}A_0=1$ is balanced. Every Koszul positively graded  
algebra is balanced in the case when $\prec$ is the full relation.
}
\end{example}

\begin{example}\label{exm5}
{\rm  
Directly from the definition it follows that if the algebra
$A$ is balanced, then the algebra $A/Ae_{\overline{\lambda}} A$ 
is balanced as well for any maximal $\lambda$.
It is also easy to see that if $A$ and $B$ are balanced, then
both $A\oplus B$ and $A\otimes_{\Bbbk}B$ are balanced.
}
\end{example}

\vspace{2mm}

\noindent
Department of Mathematics, Uppsala University, SE 471 06,
Uppsala, SWEDEN, e-mail: {\tt mazor\symbol{64}math.uu.se}

\end{document}